\documentclass[onefignum,onetabnum]{siamart250211}

\usepackage{lipsum}
\usepackage{amsfonts}
\usepackage{graphicx}
\usepackage{epstopdf}
\usepackage{algorithmicx}
\ifpdf
  \DeclareGraphicsExtensions{.eps,.pdf,.png,.jpg}
\else
  \DeclareGraphicsExtensions{.eps}
\fi

\usepackage{amssymb}
\usepackage{graphicx}
\usepackage[all]{xy}
\usepackage{color}
\usepackage{latexsym}
\usepackage{amssymb,amsmath,amstext}
\usepackage{epsfig}
\usepackage{graphics}
\usepackage{euscript}
\usepackage{ifthen}
\usepackage{wrapfig}
\usepackage{multicol}
\usepackage{multirow}
\usepackage{paralist}
\usepackage{verbatim}
\usepackage{algpseudocode}
\usepackage{placeins}
\usepackage{caption}
\usepackage{subcaption}
\usepackage{longtable}
\usepackage{lscape}
\usepackage{multicol}
\usepackage{tikz}
\usepackage{diagbox}
\usepackage{array}
\usepackage{xcolor}
\newcommand{\cmark}{\ding{51}}%
\usepackage{pifont}% http://ctan.org/pkg/pifont
\newcommand{\xmark}{\ding{55}}%

\usepackage[makeroom]{cancel}

\numberwithin{equation}{section}

\newsiamremark{remark}{Remark}
\newsiamremark{example}{Example}
\newsiamremark{hypothesis}{Hypothesis}
\crefname{hypothesis}{Hypothesis}{Hypotheses}
\newsiamthm{conjecture}{Conjecture}
\newsiamthm{problem}{Problem}

\definecolor{darkgreen}{rgb}{0,0,0}

\headers{Robust Eigenvectors of Symmetric Tensors}{T. Muller, E. Robeva, and K. Usevich}

\title{Robust Eigenvectors of Symmetric Tensors\thanks{Submitted to the editors DATE.
\funding{Elina Robeva was supported by an NSERC Discovery grant (DGECR-2020-00338).
Konstantin Usevich was supported by the ANR grant LeaFleT (ANR-19-CE23-0021).}}}

\author{Tommi Muller\thanks{University of Oxford, Oxford, UK (\email{ozitommi@gmail.com}).}
\and Elina Robeva\thanks{University of British Columbia, Vancouver, Canada (\email{erobeva@math.ubc.ca}).}
\and Konstantin Usevich\thanks{Université de Lorraine, CNRS, Nancy, France (\email{konstantin.usevich@cnrs.fr}).}}

\ifpdf
\hypersetup{
  pdftitle={Robust Eigenvectors of Symmetric Tensors},
  pdfauthor={T. Muller, E. Robeva, and K. Usevich}
}
\fi

\begin{document}

\maketitle

% REQUIRED
\begin{abstract} The {\em tensor power method} generalizes the matrix power method to higher order arrays, or tensors. Like in the matrix case, the fixed points of the tensor power method are the eigenvectors of the tensor. While every real symmetric matrix has an eigendecomposition, the vectors generating a symmetric decomposition of a real symmetric tensor are not always eigenvectors of the tensor.

In this paper we show that whenever an eigenvector {\em is} a generator of the symmetric decomposition of a symmetric tensor, then (if the order of the tensor is sufficiently high) this eigenvector is {\em robust}, i.e., it is an attracting fixed point of the tensor power method. We exhibit new classes of symmetric tensors whose symmetric decomposition consists of eigenvectors. Generalizing orthogonally decomposable tensors, we consider {\em equiangular tight frame decomposable} and {\em equiangular set decomposable} tensors.  Our main result implies that such tensors can be decomposed using the tensor power method.

\end{abstract}

% REQUIRED
\begin{keywords}
tensor power method, robust eigenvector, symmetric tensor, equiangular tight frame, tensor decomposition
\end{keywords}

% REQUIRED
\begin{MSCcodes}
15A69, 15A18, 42C15, 65F15
% 68Q25, 68R10, 68U05
\end{MSCcodes}

\section{Introduction}

With the rising demand for techniques to handle massive, high-dimensional datasets, many scientists have turned to finding adaptations of matrix algorithms to high-order arrays, known as tensors.
The main obstacle is that
determining quantities such as the rank, singular values, and eigenvalues~\cite{Hitchcock,LHLim,qi2005eigenvalues} of a general tensor is an NP-hard problem \cite{HL}. Nonetheless some heuristics have been proposed for computing such quantities~\cite{Anandkumar2014,pmlr-v40-Anandkumar15,KoBa09,shiftedPower,nie2014simax} and efficient algorithms for \textcolor{darkgreen}{tensor decomposition exist for many families of tensors  (see, e.g., \cite{Jennrich,comon2009tensor,Sorber2013,kileel2021subspace})}.
Furthermore, the decomposition, approximations, eigenvectors, and algebraic characterization have been thoroughly studied in the special case of orthogonally decomposable tensors~\cite{Anandkumar2014,BDHR,tensorTrains,kolda2015symmetric,li2018jacobi,MuHsuGoldfarb,Robeva}. 

The computation of eigenvectors and singular vectors is particularly important because it is tightly linked to the best rank-one approximation problem \cite{chen2009simax}. A recurring tool that has been used in several of the works cited here is the \textit{tensor power method}~\cite{DeLathauwer1997,DeLathauwer1995,DeLathauwer2000,kofidis2002simax}, which generalizes the well-known matrix power method.
For non-symmetric tensors, the tensor power method is globally convergent \cite{uschmajew2015pjo}, with speed of convergence established in \cite{hu2018convergence}.
For the symmetric case,  examples are known when the method does not converge at all \cite{chen2009simax,kofidis2002simax}; \textcolor{darkgreen}{the convergence has been proven only for very special cases (e.g., orthogonally decomposable tensors, see \cite{ZhaGol} and \cite[Theorem 4.1]{Anandkumar2014}.})

In this paper, we first show that if a vector in the symmetric decomposition of a symmetric tensor is an eigenvector, then for sufficiently high orders,  it is robust, i.e., it is an attracting fixed point of the tensor power method (see \Cref{mainTheorem}). \textcolor{darkgreen}{This allows one to reliably recover the eigenvectors of a tensor using the tensor power method.} We then exhibit several families of tensors whose symmetric decomposition consists of (robust) eigenvectors. Generalizing the class of orthogonally decomposable tensors, we study tensors generated by linear combinations of tensor powers of vectors which form an {\em equiangular tight frame} (\textit{ETF}), or, more generally, an \textit{equiangular set} (\textit{ES}). \textcolor{darkgreen}{ESs and ETFs are generalizations of orthonormal sets of vectors and are important objects in applied functional and harmonic analysis.}

The rest of the paper is organized as follows. In Section~\ref{backgroundSection} we provide background on tensor decompositions and the tensor power method. In Section~\ref{mainTheoremSection} we present our main result, \Cref{mainTheorem}. In Section~\ref{examples} we introduce and provide a detailed study of ETF and ES decomposable tensors; this includes the study of not only eigenvectors and their robustness, but also regions of convergence of the tensor power method. In Section~\ref{conclusionSection} we conclude with a discussion and some open problems.

\section{Background}\label{backgroundSection}
Denote by $[n]$ the set $\{1,\ldots,n\}$. We write unbolded lowercase letters for scalars in $\mathbb{F} = \mathbb{R}, \mathbb{C}$, such as $\lambda$, bolded lowercase letters for vectors, such as $\mathbf{v}$, bolded uppercase letters for matrices, such as $\mathbf{M}$, and script letters for $d$-tensors where $d \geq 3$, such as $\mathcal{T}$. An order $d$ tensor with dimensions $n_1,\ldots,n_d$ is an element $\mathcal{T} \in \mathbb{F}^{n_1 \times \cdots \times n_d}$. The $(i_1,\ldots,i_d)$-th entry of the tensor $\mathcal{T}$ will be denoted by $\mathcal{T}_{i_1,\ldots,i_d}$ where $i_1 \in [n_1],\ldots,i_d \in [n_d]$. An order $d$ tensor $\mathcal{T} \in \mathbb{F}^{n \times \cdots \times n}$ is said to be \textit{symmetric} if for all permutations $\sigma \in S_d$
of~$[d]$,
$$\mathcal{T}_{i_1,\ldots,i_d} = \mathcal{T}_{i_{\sigma(1)},\ldots,i_{\sigma(d)}}.$$
We denote the set of all symmetric tensors of order $d$ and dimension $n$ by $S^d(\mathbb{F}^n)$.
\begin{definition}\label{symmetricDecompositionDef}
A \textit{symmetric decomposition} of a symmetric tensor $\mathcal{T}\in S^d(\mathbb{F}^n)$ is an expression of $\mathcal{T}$ of the form
\begin{equation}\label{tensorGeneration}
    \mathcal{T} = \sum_{i=1}^r \lambda_i \mathbf{v}_i^{\otimes d},
\end{equation}
where $\lambda_1,\dots, \lambda_r\in \mathbb{F}$, $\mathbf{v}_1,\dots,\mathbf{v}_r\in\mathbb{F}^d$ are unit-norm vectors, and $\mathbf{v}^{\otimes d} = \mathbf{v}\otimes\cdots\otimes \mathbf{v}$ ($d$-times) is a \textit{symmetric rank-one tensor}. We say that $\mathcal{T}$ is \textit{generated} by the vectors $\mathbf{v}_1,\ldots,\mathbf{v}_r$ and the coefficients $\lambda_1,\ldots,\lambda_r$. The smallest $r$ for which such a decomposition exists is called the \textit{symmetric rank} of $\mathcal{T}$.
\textcolor{darkgreen}{In what follows we use the shorthand notation}
\[
\textcolor{darkgreen}{\mathbf{V} = \begin{pmatrix} | & & |\\
\mathbf{v}_1 & \cdots & \mathbf{v}_r\\
| & & |
\end{pmatrix}}
\]
\textcolor{darkgreen}{whenever we deal with the symmetric decomposition \eqref{tensorGeneration}.}
\end{definition}
The rank of any \textcolor{darkgreen}{$n\times n$} symmetric matrix is always at most $n$. For tensors $\mathcal{T} \in S^d(\mathbb{C}^n)$, this is not the case since the rank can be much larger. The Alexander-Hirschowitz Theorem \cite{BRAMBILLA20081229} states that, with probability $1$, the symmetric rank of a random tensor $\mathcal{T} \in S^d(\mathbb{C}^n)$  (drawn from an absolutely continuous probability distribution) is $\lfloor \frac1n\binom{n+d-1}d\rfloor$  except for a few special values of $d$ and $n$ where the rank is $1$ more than this number.

A vector $\mathbf{v} \in \mathbb{F}^n$ is an \textit{eigenvector} of $\mathcal{T}$ with \textit{eigenvalue} $\mu \in \mathbb{F}$ if
$$\mathcal{T} \cdot \mathbf{v}^{d-1} = \mu \mathbf{v},$$
where $\mathcal{T} \cdot \mathbf{v}^{d-1}$ is a vector defined by \textit{contracting} $\mathcal{T}$ by $\mathbf{v}$ along all of its modes except for one, i.e. the $i$-th entry of $\mathcal{T} \cdot \mathbf{v}^{d-1}$ is

$$(\mathcal{T} \cdot \mathbf{v}^{d-1})_i = \sum_{i_1,\dots, i_{d-1}=1}^n \mathcal{T}_{i_1\dots i_{d-1} i}{v}_{i_1} \cdots {v}_{i_{d-1}}.$$
Since $\mathcal{T}$ is symmetric, it does not matter which $d-1$ modes of $\mathcal{T}$ we contract. The eigenvectors of $\mathcal{T}$ are the fixed points (up to sign) of an \textcolor{darkgreen}{iterative} method called the \textit{tensor power method} \textcolor{darkgreen}{ (introduced in \cite{DeLathauwer1995},
\cite[Ch. 5]{DeLathauwer1997} for symmetric and non-symmetric case)}, given by
$$\mathbf{x}_{\textcolor{darkgreen}{k}} \mapsto  \textcolor{darkgreen}{\mathbf{x}_{k+1} = }  \frac{\mathcal{T} \cdot \mathbf{x}_k^{d-1}}{\|\mathcal{T} \cdot \mathbf{x}_k^{d-1}\|}.$$
Yet another important characterization of the eigenvectors is that they are the critical points of the symmetric best rank-one approximation problem:
\begin{equation}\label{rankOneApprox}
\min_{c, \mathbf{v}} \|\mathcal{T}  - c \mathbf{v}^{\otimes d}\|^2_F,
\end{equation}
see e.g., \cite[\S6]{chen2009simax} for a related discussion.
\textcolor{darkgreen}{In particular, some of the eigenvectors can be numerically found by the conventional optimization-bazed algorithms for rank-one approximation of tensors.}
Note that while it is known that a non-symmetric best rank-one approximation of a symmetric tensor can be always chosen symmetric \cite{friedland2013best}, this does not give us information about \emph{all}  critical points; in particular the results about the convergence of the non-symmetric power method \cite{uschmajew2015pjo} cannot be applied.

We call the vector $\mathbf{x}_0$ an \textit{initializing vector} of the tensor power method. Note that we are interested in real tensors $\mathcal{T}$ and their real eigenvectors and investigate the convergence behavior of the tensor power method for real, non-zero initializing vectors. We are also not interested in eigenvectors of $\mathcal{T}$ that have eigenvalue $0$, since in that case, the tensor power method is not applicable. A \textit{robust eigenvector} of $\mathcal{T}$ is an eigenvector $\mathbf{v}$ that is an attracting fixed point of the tensor power method, i.e. there exists an $\epsilon > 0$ such that the tensor power method converges to $\mathbf{v}$ for all initializing vectors $\mathbf{x}_0 \in B_{\epsilon}(\mathbf{v})$ in the ball of radius $\epsilon$ centered at $\mathbf{v}$. This means that an eigenvector is robust if it can be reliably obtained from the tensor power method.

A tensor $\mathcal{T}$ is said to be \textit{orthogonally decomposable}, or \textit{odeco}, if it has a symmetric decomposition of the form
\begin{equation}\label{odecoTensor}
    \mathcal{T} = \sum_{i=1}^n\lambda_i \mathbf{v}_i^{\otimes d},
\end{equation}
where $\mathbf{v}_1,\ldots,\mathbf{v}_n$ form an orthonormal basis of $\mathbb{R}^n$. Since there are at most $n$ orthogonal vectors in $\mathbb{R}^n$, the symmetric rank of an odeco tensor is at most $n$. Odeco tensors have been thoroughly characterized and display a number of remarkable properties~\cite{Anandkumar2014,BDHR, Robeva, RobevaSeigal}. One, of interest here, is that the robust eigenvectors of an odeco tensor $(\ref{odecoTensor})$ are precisely $\mathbf{v}_1,\ldots,\mathbf{v}_n$ \cite{Anandkumar2014}.

This brings a few important points. In general, a symmetric tensor with a symmetric decomposition $(\ref{tensorGeneration})$, firstly, may not have $\mathbf{v}_j$ as an eigenvector for some $j$, secondly, may have eigenvectors that are not robust, and lastly, may have robust eigenvectors that are not one of the $\mathbf{v}_j$'s. In comparison, the only robust \textcolor{darkgreen}{eigenvector} of a generic symmetric matrix is the one whose eigenvalue is largest in absolute value. Additionally, unlike symmetric matrices, symmetric tensors can have several robust eigenvectors.

In the following Section~\ref{mainTheoremSection}, we present our main result which essentially says that if a term $\mathbf{v}_i^{\otimes d}$ is part of the symmetric decomposition~$(\ref{tensorGeneration})$ of a symmetric tensor $\mathcal T$, then, for $d$ sufficiently large, if $\mathbf{v}_i$ is an eigenvector, then it is robust. In Section~\ref{examples}, we introduce a family of tensors, called \textit{equiangular tensors}, which generalize odeco tensors. These tensors share the property that the vectors $\mathbf{v}_i$ in the symmetric decomposition $(\ref{tensorGeneration})$ are eigenvectors. We apply our main result to study the robustness of these eigenvectors. We leave the study of robust eigenvectors that do not generate the decomposition of the symmetric tensor as an open problem in the conclusion.

\section{Main Theorem}\label{mainTheoremSection}
We now proceed to our main result which gives a condition on when an eigenvector is robust.
\begin{theorem}\label{mainTheorem}
For $d \in \mathbb{N}$, let $\mathcal{T}_d \in S^d(\mathbb{R}^n)$ be a tensor with symmetric \textcolor{darkgreen}{(not necessarily minimal)} decomposition
\begin{equation}{\label{rank1}}
    \mathcal{T}_d = \sum_{i = 1}^r \lambda_i\mathbf{v}_i^{\otimes{d}},
\end{equation}
with $\|\mathbf{v}_i\| = 1$ for all $i$. Then there exists a $D \in \mathbb{N}$ such that for all $d \geq D$, if $\mathbf{v}_j$ is an eigenvector of $\mathcal{T}_d$ with \textcolor{darkgreen}{a} non-zero eigenvalue, then $\mathbf{v}_j$ is a robust eigenvector of $\mathcal{T}_d$.
\end{theorem}
As we will see in the sections to follow, this result allows us to use the tensor power method in order to decompose certain classes of tensors.

The following lemma is used in the proof of \Cref{mainTheorem}.
\begin{lemma}\textup{\cite[Theorem 3.5]{Rheinboldt1998}}\label{powermethodtheorem}
Let $\mathbf{x_*} \in \mathbb{R}^n$ be a fixed point of a $C^1(U,\mathbb{R}^n)$ function $\phi:~U\rightarrow~\mathbb{R}^n$ where $U \subseteq \mathbb{R}^n$ is an open set, and let $\mathbf{J}: U \rightarrow \mathbb{R}^{n \times n}$ be the Jacobian matrix of $\phi$. Then $\mathbf{x_*}$ is an attracting fixed point of the iterative method $\mathbf{x}_{k+1} = \phi(\mathbf{x}_k)$ if $\rho(\mathbf{J}(\mathbf{x_*})) < 1$, where $\rho(\mathbf{J}(\mathbf{x_*}))$ is the spectral radius of the matrix $\mathbf{J}(\mathbf{x_*})$. Furthermore, if $\rho(\mathbf{J}(\mathbf{x_*})) > 0$, then for $\mathbf{x}_0$ sufficiently close to $\mathbf{x_*}$, the rate of convergence of this iterative method is linear.
\end{lemma}
We will also need the following lemma about the structure of the Jacobian matrix of the tensor power method iteration.
\begin{lemma}\label{jacobianLemma}
Let $\mathcal{T}_d \in S^d(\mathbb{R}^n)$ and let $\phi: U \rightarrow \mathbb{R}^n$ be the tensor power method iteration map
\begin{equation}\label{phiFormula_general}
\phi(\mathbf{x}) = \frac{\mathcal{T}_d \cdot \mathbf{x}^{d-1}}{\|\mathcal{T}_d \cdot \mathbf{x}^{d-1}\|}
\end{equation}
where $U \subseteq \mathbb{R}^n$ is an open set. Assume that the vector $\mathbf{v} \in \mathbb{R}^n$ is a unit-norm eigenvector of $\mathcal{T}_d$ with non-zero eigenvalue $\mu \in \mathbb{R}$. Then the Jacobian matrix of $\phi$ at $\mathbf{v}$, $\mathbf{J}(\mathbf{v})$, is symmetric and has the following form:
$$\mathbf{J}(\mathbf{v}) = \frac{(d-1)}{\mu}\left(\mathcal{T} \cdot \mathbf{v}^{d-2} -  \mu \mathbf{v}\mathbf{v}^{\top}\right).$$
\end{lemma}
\begin{proof}
Denote $\phi_1(\mathbf{x}) = \frac{\mathbf{x}}{\left(\mathbf{x}^\top \mathbf{x}\right)^{\frac{1}{2}}}$ and $\phi_2(\mathbf{x}) = \mathcal{T}_d \cdot \mathbf{x}^{d-1}$ so that $\phi(\mathbf{x}) = \phi_1(\phi_2(\mathbf{x}))$. Then

$$\phi_1^\prime(\mathbf{x}) = \frac{\mathbf{x}^\top\mathbf{x} \mathbf{I}_{n \times n} - \mathbf{x}\mathbf{x}^\top}{\left(\mathbf{x}^\top\mathbf{x}\right)^{\frac{3}{2}}}$$
and
$$\phi_2^\prime(\mathbf{x}) = (d-1)\left(\mathcal{T} \cdot \mathbf{x}^{d-2}\right).$$
Next, we can express
$$\phi_1^\prime(\phi_2(\mathbf{x})) =
\frac{\|\mathcal{T} \cdot \mathbf{x}^{d-1}\|^2 \mathbf{I}_{n\times n} - (\mathcal{T} \cdot \mathbf{x}^{d-1})(\mathcal{T} \cdot \mathbf{x}^{d-1})^{\top}}{ (\|\mathcal{T} \cdot \mathbf{x}^{d-1}\|)^{3}}$$
and therefore, by the chain rule,
\begin{align*}
\mathbf{J}(\mathbf{x}) &= \phi^\prime(\mathbf{x}) = \phi_1^\prime(\phi_2(\mathbf{x}))\phi_2^\prime(\mathbf{x}) \\
&=  (d-1)\frac{\|\mathcal{T} \cdot \mathbf{x}^{d-1}\|^2(\mathcal{T} \cdot \mathbf{x}^{d-2}) - (\mathcal{T} \cdot \mathbf{x}^{d-1})(\mathcal{T} \cdot \mathbf{x}^{d-1})^{\top}(\mathcal{T} \cdot \mathbf{x}^{d-2})}{(\|\mathcal{T} \cdot \mathbf{x}^{d-1}\|)^{3}}.
\end{align*}
Now let us evaluate the expression at the unit-norm eigenvector $\mathbf{v}$ corresponding to an eigenvalue $\mu \neq 0$ (i.e., satisfying $\mathcal{T} \cdot \mathbf{v}^{d-1} = \mu \mathbf{v}$). Then we have that
$$\mathbf{J}(\mathbf{v}) = 
(d-1)\frac{\mu^2 (\mathcal{T} \cdot \mathbf{v}^{d-2}) -  \mu^2 \mathbf{v} \mathbf{v}^{\top} (\mathcal{T} \cdot \mathbf{v}^{d-2})}{ \mu^{3}} = 
\frac{(d-1)}{\mu}\left(\mathcal{T} \cdot \mathbf{v}^{d-2} -  \mu \mathbf{v}\mathbf{v}^{\top}\right),$$
where we used the fact that $\mathbf{v}^{\top} (\mathcal{T} \cdot \mathbf{v}^{d-2}) = (\mathcal{T} \cdot \mathbf{v}^{d-1})^\top = \mu\mathbf{v}^\top$.
\end{proof}

We now proceed with the proof of our theorem.
\begin{proof}[Proof of \Cref{mainTheorem}]
We may assume that no $\lambda_i$ is $0$ and no two vectors $\mathbf{v}_k$ and $\mathbf{v}_\ell$ are colinear, or else we may rewrite $\mathcal{T}_d$ as a sum of a smaller number $r$ of symmetric rank-1 tensors.

Contracting $\mathcal{T}_d$ with $d-1$ copies of $\mathbf{v}_j$, since $\mathbf{v}_j$ is an eigenvector of $\mathcal{T}_d$ with eigenvalue $\mu_{j,d} \neq 0$, we have
\begin{equation}\label{C-coeff}
\begin{split}
\mathcal{T}_d \cdot \mathbf{v}_j^{d-1}& = \sum_{i = 1}^r \lambda_i\langle \mathbf{v}_i, \mathbf{v}_j \rangle^{d-1} \mathbf{v}_i = \lambda_j\mathbf{v}_j + \sum_{i \in [r]\setminus\{j\}} \lambda_i\alpha_{i,j}^{d-1} \mathbf{v}_i  \\
 &= \mathbf{V}\mathbf{\Lambda}\left(\mathbf{V}^\top\mathbf{v}_j\right)^{\odot(d-1)} = \mu_{j,d}\mathbf{v}_j,\,
\end{split}
\end{equation}
where $\alpha_{i,j} := \langle \mathbf{v}_i, \mathbf{v}_j \rangle$, and contracting $\mathcal{T}_d$ with $d-2$ copies of $\mathbf{v}_j$, we have
\[
\mathcal{T}_d \cdot \mathbf{v}_j^{d-2} = \sum_{i = 1}^r \lambda_i \langle \mathbf{v}_i, \mathbf{v}_j \rangle^{d-2} \mathbf{v}_i\mathbf{v}_i^\top = \mathbf{V}\mathbf{\Lambda}\mathbf{D}(\textcolor{darkgreen}{\mathbf{v}_j})\mathbf{V}^\top,
\]
where $\mathbf{\Lambda} = \text{diag}\left(\lambda_1,\ldots,\lambda_r\right)$, $\mathbf{D}(\mathbf{x}) = \text{diag}\left((\mathbf{V}^\top\mathbf{x})^{\odot(d-2)}\right)$, and $\mathbf{x}^{\odot m} = (x_1^m,\ldots,x_n^m)^\top$ is the $m$-th Hadamard power of the vector $\mathbf{x}$. Hence, by \Cref{jacobianLemma}, we have that the Jacobian matrix of $\phi$ at $\mathbf{v}_j$ is
\begin{equation}\label{jacobianForm}
\mathbf{J}(\mathbf{v}_j) = \frac{d-1}{\mu_{j,d}} \left(\mathbf{V}\mathbf{\Lambda}\mathbf{D}(\mathbf{v}_j)\mathbf{V}^\top - \mu_{j,d}\mathbf{v}_{j}\mathbf{v}_{j}^{\top}\right).
\end{equation}
Now multiplying both sides of \eqref{C-coeff} by $\mathbf{v}_j^\top$ on the right, we have
$$\lambda_j\mathbf{v}_j\mathbf{v}_j^\top + \sum_{i \in [r]\setminus\{j\}} \lambda_i\alpha_{i,j}^{d-1} \mathbf{v}_i\mathbf{v}_j^\top = \mu_{j,d}\mathbf{v}_j\mathbf{v}_j^\top,$$
and hence
\begin{equation}\label{substitute-outer}
    \lambda_j\mathbf{v}_j\mathbf{v}_j^\top - \mu_{j,d}\mathbf{v}_j\mathbf{v}_j^\top = - \sum_{i \in [r]\setminus\{j\}} \lambda_i\alpha_{i,j}^{d-1} \mathbf{v}_i\mathbf{v}_j^\top.
\end{equation}
Next, we are going to bound the spectral radius of $\mathbf{J}(\mathbf{v}_j)$, which is equal to $\|\mathbf{J}(\mathbf{v}_j)\|_2$ because $\mathbf{J}(\mathbf{v}_j)$ is symmetric. Due to \eqref{substitute-outer}, we can express
\begin{align*}
\mathbf{V}\mathbf{\Lambda}\mathbf{D}(\mathbf{v}_j)\mathbf{V}^\top  - 
\mu_{j,d}\mathbf{v}_{j} \mathbf{v}^{\top}_{j} & = \sum_{i \in [r]} \lambda_i\alpha_{i,j}^{d-2}\mathbf{v}_i\mathbf{v}_i^\top - \mu_{j,d}\mathbf{v}_{j} \mathbf{v}^{\top}_{j} \\
&= \sum_{i \in [r]\setminus\{j\}} \lambda_i\alpha_{i,j}^{d-2}\mathbf{v}_i\mathbf{v}_i^\top - \sum_{i \in [r]\setminus\{j\}} \lambda_i\alpha_{i,j}^{d-1} \mathbf{v}_i\mathbf{v}_j^\top\\
 &= 
\sum_{i \in [r]\setminus\{j\}} \lambda_i\alpha_{i,j}^{d-2}\mathbf{v}_i(\mathbf{v}_i- \alpha_{i,j} \mathbf{v}_j)^\top \\
&=\left(\sum_{i \in [r]\setminus\{j\}} \lambda_i\alpha_{i,j}^{d-2}\mathbf{v}_i\mathbf{v}^{\top}_i\right) (\mathbf{I}- \mathbf{v}_j \mathbf{v}_j^\top).
\end{align*}
Therefore, since $\mathbf{I}- \mathbf{v}_j \mathbf{v}_j^\top$ is an orthogonal projector with $\|\mathbf{I}- \mathbf{v}_j \mathbf{v}_j^\top\|_2 = 1$, and by submultiplicativity of the spectral norm, we get
\begin{align}
\rho(\mathbf{J}(\mathbf{v}_j)) &= \left|\frac{d-1}{\mu_{j,d}}\right| \|\mathbf{V}\mathbf{\Lambda}\mathbf{D}(\mathbf{v}_j)\mathbf{V}^\top  - 
\mu_{j,d}\mathbf{v}_{j} \mathbf{v}^{\top}_{j}\|_2 \le \left|\frac{d-1}{\mu_{j,d}}\right| \left\|\sum_{i \in [r]\setminus\{j\}} \lambda_i\alpha_{i,j}^{d-2}\mathbf{v}_i\mathbf{v}^{\top}_i\right\|_2  \label{generalSpectralBound} \\ \nonumber
&\le \left|\frac{d-1}{\mu_{j,d}}\right| {\sum_{i \in [r]\setminus\{j\}}}  \left\|\lambda_i\alpha_{i,j}^{d-2}\mathbf{v}_i\mathbf{v}^{\top}_i\right\|_2 = \left|\frac{d-1}{\mu_{j,d}}\right| {\sum_{i \in [r]\setminus\{j\}}}  |\lambda_i\alpha_{i,j}^{d-2}| \\ \nonumber
& \le   \left|\frac{d-1}{\mu_{j,d}}\right|(r-1)\left(\max_{i \in [r]\setminus\{j\}}\left|\lambda_i\right|\right)\left(\max_{i \in [r]\setminus\{j\}}\left|\alpha_{i,j}\right|\right)^{d-2},
\end{align}
where we used the triangle inequality and the Cauchy-Schwarz inequalities. 

Note that since no two vectors $\mathbf{v}_k$ and $\mathbf{v}_\ell$ are colinear, $|\alpha_{i,j}| < 1$ for all $i \in [r]\setminus\{j\}$. Therefore, rearranging $(\ref{C-coeff})$ and applying the triangle inequality,
$$|\mu_{j,d} - \lambda_j| \leq \sum_{i \in [r]\setminus\{j\}} |\lambda_i||\alpha_{i,j}|^{d-1},$$
we see that $\mu_{j,d}$ converges to $\lambda_j$ as $d$ becomes large, \textcolor{darkgreen}{thus the denominator in \eqref{generalSpectralBound} is separated from $0$}. This, \textcolor{darkgreen}{together with the fact that $|\alpha_{i,j}| < 1$ for $i \neq j$, shows that  $\rho(\mathbf{J}(\mathbf{v}_j)) \to 0$ and hence is less than $1$ for sufficiently large $d$ and hence}  the result follows by \Cref{powermethodtheorem}.
\end{proof}

\section{Equiangular Tensors}\label{examples}
The main result \Cref{mainTheorem} compels us to find sufficient conditions for when a generating vector of a symmetric tensor $\mathcal{T}$ in $(\ref{tensorGeneration})$ is an eigenvector. We will see that \textcolor{darkgreen}{a tensor} $\mathcal{T}$ generated by a certain class of vectors will have this property.
\subsection{Equiangular sets}
\begin{definition}
An \textit{equiangular set} (\textit{ES}) is a collection of vectors $\mathbf{v}_1,\ldots,\mathbf{v}_r \in \mathbb{R}^n$ with $r \geq n$ \textcolor{darkgreen}{for which} there exists $
\alpha\in\mathbb R$ such that
\begin{equation}\label{alphaEquation}
    \alpha = |\langle \mathbf{v}_i, \mathbf{v}_j \rangle|, \,\forall i\neq j \quad\quad\text{ and } \quad\quad\|\mathbf v_i\|=1, \,\forall i.
\end{equation}
\end{definition}
Note that if $\mathbf{v}_1,\ldots,\mathbf{v}_r$ form an ES, and if $\sigma_{i,j} = \text{sgn}(\langle \mathbf{v}_i, \mathbf{v}_j \rangle) \in \{1,-1\}$ for $i \neq j$, then 
\begin{equation}\label{VtV}
    \begin{pmatrix} - & \mathbf{v}_1 & -\\
& \vdots & \\
- & \mathbf{v}_r & -
\end{pmatrix}\begin{pmatrix} | & & |\\
\mathbf{v}_1 & \cdots & \mathbf{v}_r\\
| & & |
\end{pmatrix} = \begin{pmatrix}
1 & \sigma_{1,2}\alpha & \hdots & \sigma_{1,r-1} \textcolor{darkgreen}{\alpha} & \sigma_{1,r}\alpha \\
\sigma_{2,1}\alpha & 1 & \hdots & \sigma_{2,r-1} \textcolor{darkgreen}{\alpha} & \sigma_{2,r}\alpha \\
\vdots & \vdots & \ddots & \vdots & \vdots \\
\sigma_{r-1,1}\alpha & \sigma_{r-1,2}\alpha & \hdots & 1 & \sigma_{r-1,r}\alpha \\
\sigma_{r,1}\alpha & \sigma_{r,2}\alpha & \hdots & \sigma_{r,r-1}\alpha & 1
\end{pmatrix}.
\end{equation}
ESs correspond to sets of lines in $\mathbb{R}^n$ passing through the origin such that the angle between every pair of lines is the same. Determining the maximum number of equiangular lines in $\mathbb{R}^n$ for each $n$ is an old problem that has recently seen significant progress by \cite{equiangularLineBEST}, who determined an asymptotically tight upper bound.

We call a tensor $(\ref{tensorGeneration})$ generated by vectors from an ES \textit{equiangular set decomposable}, or \textit{equiangular} for short. We begin with some general results on equiangular tensors.
\begin{theorem}\label{kernelConditionLemma}
Let $\mathcal{T}$ be a tensor generated by an ES $\mathbf{v}_1,\ldots,\mathbf{v}_r \in \mathbb{R}^n$ and coefficients $\lambda_1,\ldots,\lambda_r \in \mathbb{R}$. If for some $j\in[r]$, there exists $\textcolor{darkgreen}{C_j} \in \mathbb{R}$ such that
\begin{equation}\label{kernelCondition}
    (\lambda_1\sigma_{1,j}^{d-1},\ldots,\lambda_{j-1}\sigma_{j-1,j}^{d-1},\textcolor{darkgreen}{C_j},\lambda_{j+1}\sigma_{j+1,j}^{d-1},\ldots,\lambda_r\sigma_{r,j}^{d-1}) \in Ker(\mathbf{V})
\end{equation}
where $\mathbf{V} \in \mathbb{R}^{n \times r}$ is the matrix whose columns are $\mathbf{v}_1,\ldots,\mathbf{v}_r$, then $\mathbf{v}_j$ is an eigenvector. 
In particular, if $d$ is odd and $(\lambda_1,\ldots, \lambda_r)\in Ker(\mathbf V)$, then, all of $\mathbf v_1,\ldots, \mathbf v_r$ are eigenvectors of $\mathcal T$.
Furthermore, in this case, all of these vectors are robust eigenvectors if
\begin{equation}\label{lambdasInKernelBound}
    \frac{\|\mathbf{V}\textcolor{darkgreen}{\mathbf{\Lambda}}\mathbf{V}^\top\|_2\alpha^{d-2}(d-1)}{\left(\min_{i \in [r]} |\lambda_i|\right)(1 - \alpha^{d-1})} < 1,
\end{equation}
which always holds when $d$ is large enough,  \textcolor{darkgreen}{where $\mathbf{\Lambda} = \text{diag}\left(\lambda_1,\ldots,\lambda_r\right)$.}
% \textcolor{red}{TODO: remind what is Lambda.}
\end{theorem}

\begin{proof}
We observe that for $j \in [r]$,
$$\mathcal{T} \cdot \mathbf{v}_j^{d-1} = \sum_{i = 1}^r \lambda_i\langle \mathbf{v}_i, \mathbf{v}_j \rangle^{d-1} \mathbf{v}_i = \lambda_j\mathbf{v}_j + \alpha^{d-1}\sum_{i \in [r]\setminus\{j\}} \lambda_i\sigma_{i,j}^{d-1} \mathbf{v}_i,$$
and hence if $(\ref{kernelCondition})$ holds, then
$$\lambda_j\mathbf{v}_j + \alpha^{d-1}\sum_{i \in [r]\setminus\{j\}} \lambda_i\sigma_{i,j}^{d-1} \mathbf{v}_i = \lambda_j\mathbf{v}_j - \alpha^{d-1}\mu_j\mathbf{v}_j = (\lambda_j - \alpha^{d-1}\mu_j)\mathbf{v}_j.$$
Therefore, $\mathbf v_j$ is an eigenvector of $\mathcal T$.

Now if $d$ is odd and $(\lambda_1,\ldots,\lambda_r) \in Ker(\mathbf{V)}$, we have
\begin{align*}
\mathcal{T} \cdot \mathbf{v}^{d-1}_j &= \sum_{i=1}^{r}\lambda_i\langle \mathbf{v}_i, \mathbf{v}_j\rangle^{d-1}\mathbf{v}_i = \lambda_j\mathbf{v}_j + \sum_{i \in [r]\setminus\{j\}} \lambda_i(\sigma_{i,j}\alpha)^{d-1}\mathbf{v}_i \\
&= \lambda_j\mathbf{v}_j + \alpha^{d-1}\sum_{i \in [r]\setminus\{j\}} \lambda_i\mathbf{v}_i \\
&= \lambda_j\mathbf{v}_{j} + \alpha^{d-1}(-\lambda_j\mathbf{v}_{j}) = \lambda_j(1 - \alpha^{d-1})\mathbf{v}_{j},
\end{align*}
i.e., all of $\mathbf v_1,\ldots, \mathbf v_r$ are eigenvectors. In addition, using inequality \eqref{generalSpectralBound},
\begin{align*}
\rho(\mathbf{J}(\mathbf{v}_j)) &  \le \frac{d-1}{|\lambda_j|(1 - \alpha^{d-1})}  \alpha^{d-2} \left\|\sum_{i \in [r]\setminus\{j\}} \textcolor{darkgreen}{\lambda_i}\mathbf{v}_i\mathbf{v}^{\top}_i\right\|_2 \\
&\leq \frac{d-1}{|\lambda_j|(1 - \alpha^{d-1})}\alpha^{d-2}\|\mathbf{V} \textcolor{darkgreen}{\mathbf{\Lambda}}\mathbf{V}^\top\|_2 \leq \frac{\|\mathbf{V}\textcolor{darkgreen}{\mathbf{\Lambda}}\mathbf{V}^\top\|_2\alpha^{d-2}(d-1)}{\left(\min_{i \in [r]} |\lambda_i|\right)(1 - \alpha^{d-1})}.
\end{align*}
When the above quantity is less than $1$, $\mathbf{v}_j$ is a robust eigenvector, for all $j\in[r]$.
\end{proof}

\subsection{Equiangular tight frames}
\begin{definition}
An ES is an \textit{equiangular tight frame} (\textit{ETF}) if, in addition,
\begin{equation}\label{VVt}
    \begin{pmatrix} | & & |\\
\mathbf{v}_1 & \cdots & \mathbf{v}_r\\
| & & |
\end{pmatrix}\begin{pmatrix} - & \mathbf{v}_1 & -\\
& \vdots & \\
- & \mathbf{v}_r & -
\end{pmatrix} = \frac{r}{n} \mathbf{I}_{n}
\end{equation}
where $\mathbf I_{n} \in \mathbb{R}^{n \times n}$ is the identity matrix.
\end{definition}
Suppose $\mathbf{v}_1,\ldots,\mathbf{v}_r$ form an ETF. Then a number of additional results can be deduced. If $\mathbf{u}_1,\ldots,\mathbf{u}_r \in \mathbb{R}^n$ with $r \geq n$ is a collection of \textcolor{darkgreen}{unit} vectors, then the following always holds
\begin{equation}\label{welchBound}
    \max_{\substack{i,j \in [r] \\ i \neq j}} |\langle \mathbf{u}_i, \mathbf{u}_j \rangle| \geq \sqrt{\frac{r-n}{n(r-1)}}
\end{equation}
with equality if and only if $\mathbf{u}_1,..,\mathbf{u}_r$ is an ETF \cite{etfWelchBound}. Thus, $\alpha = \sqrt{\frac{r-n}{n(r-1)}}$ in $(\ref{alphaEquation})$. Furthermore, the matrix $\mathbf{V}^\top\mathbf{V} \in \mathbb{R}^{r \times r}$ in $(\ref{VtV})$, known as the \textit{Gram matrix}, has rank $n$ (Proposition 3, \cite{gramMatrixRank}). The Gram matrix gives a canonical representation of an equiangular tight frame. This results in a one-to-one correspondence between ETFs up to orthogonal transformation and their corresponding Gram matrix \cite{WALDRON20092228}, which means that $\mathbf{V}$ also has rank $n$.

ETFs with $r$ vectors in $\mathbb{R}^n$ do not exist for many values of $r$ and $n$, making ETFs quite rare \cite{ETFexistence}. Nonetheless, they have attracted a wide interest for a number of reasons. ETFs are a natural generalization of orthonormal sets of vectors where the number of vectors in the set is allowed to exceed the dimension of the space they lie in. ETFs minimize the maximum coherence between the vectors, attaining equality in what is known as the \textit{Welch bound} $(\ref{welchBound})$. ETFs can also be formulated for $\mathbb{C}^n$ and have found numerous applications in signal processing~\cite{steinerFrames}, coding theory~\cite{grassmannFrames}, and quantum information processing~\cite{quantumFrames}.

When a tensor is generated by the vectors in an ETF,  it is called an {\em ETF decomposable} tensor. Such tensors are a special case of {\em fradeco} tensors, which were studied in~\cite{OEDING2016125}. \textcolor{darkgreen}{In contrast to the case of orthogonally decomposable tensors, the complete equations determining fradeco tensors are not known. However, partial progress has been made in \cite{OEDING2016125} where a subset of the complete equations has been determined, in special cases.}

\textcolor{darkgreen}{ETFs are a special case of ESs, and hence, in addition to \Cref{kernelConditionLemma}, we also obtain the following sufficiency criteria for robust eigenvectors for ETF tensors.}
\begin{theorem}\label{constanteigenvector}
If $\mathbf{v}_1,\ldots,\mathbf{v}_r \in \mathbb{R}^n$ form an ETF, then $\displaystyle \sum_{i \in [r]\setminus\{j\}} \sigma_{ij}\mathbf{v}_i = C\mathbf{v}_j$ for some $C \in \mathbb{R}$, for all $j \in [r]$. In particular, all of $\mathbf{v}_1,\ldots,\mathbf{v}_r$ are eigenvectors of the tensor
\begin{equation}\label{allOnes}
    \mathcal{T} = \sum_{i = 1}^r \mathbf{v}_i^{\otimes d}
\end{equation}
when $d$ is even. Furthermore, in this case, all of these vectors are robust eigenvectors if
\begin{equation}\label{allOnesBound}
    \frac{\frac{r}{n}\alpha^{d-2}(d-1)}{1 + \alpha^{d-2}\left(\frac{r}{n}-1\right)} < 1,
\end{equation}
which always holds when $d$ is large enough.
\end{theorem}
\begin{proof}
Starting with the Gram matrix of the ETF,
$$\begin{pmatrix}
1 & \sigma_{1,2}\alpha & \hdots & \sigma_{1,r-1}\textcolor{darkgreen}{\alpha} & \sigma_{1,r}\alpha \\
\sigma_{2,1}\alpha & 1 & \hdots & \sigma_{2,r-1}\textcolor{darkgreen}{\alpha} & \sigma_{2,r}\alpha \\
\vdots & \vdots & \ddots & \vdots & \vdots \\
\sigma_{r-1,1}\alpha & \sigma_{r-1,2}\alpha & \hdots & 1 & \sigma_{r-1,r}\alpha \\
\sigma_{r,1}\alpha & \sigma_{r,2}\alpha & \hdots & \sigma_{r,r-1}\alpha & 1
\end{pmatrix} = \mathbf{V}^\top\mathbf{V}$$
we subtract the identity matrix on both sides of the equation to obtain
$$\begin{pmatrix}
0 & \sigma_{1,2}\alpha & \hdots & \sigma_{1,r-1} \textcolor{darkgreen}{\alpha} & \sigma_{1,r}\alpha \\
\sigma_{2,1}\alpha & 0 & \hdots & \sigma_{2,r-1} \textcolor{darkgreen}{\alpha} & \sigma_{2,r}\alpha \\
\vdots & \vdots & \ddots & \vdots & \vdots \\
\sigma_{r-1,1}\alpha & \sigma_{r-1,2}\alpha & \hdots & 0 & \sigma_{r-1,r}\alpha \\
\sigma_{r,1}\alpha & \sigma_{r,2}\alpha & \hdots & \sigma_{r,r-1}\alpha & 0
\end{pmatrix} = \mathbf{V}^\top\mathbf{V} - \mathbf{I}_{r \times r}.$$
Multiplying on the left of both sides of the equation by $\mathbf{V}$,
$$\mathbf{V}\begin{pmatrix}
0 & \sigma_{1,2}\alpha & \hdots & \sigma_{1,r-1} \textcolor{darkgreen}{\alpha} & \sigma_{1,r}\alpha \\
\sigma_{2,1}\alpha & 0 & \hdots & \sigma_{2,r-1} \textcolor{darkgreen}{\alpha} & \sigma_{2,r}\alpha \\
\vdots & \vdots & \ddots & \vdots & \vdots \\
\sigma_{r-1,1}\alpha & \sigma_{r-1,2}\alpha & \hdots & 0 & \sigma_{r-1,r}\alpha \\
\sigma_{r,1}\alpha & \sigma_{r,2}\alpha & \hdots & \sigma_{r,r-1}\alpha & 0
\end{pmatrix} = \mathbf{V}\mathbf{V}^\top\mathbf{V} - \mathbf{V}\mathbf{I}_{r \times r} = (\mathbf{V}\mathbf{V}^\top)\mathbf{V} - \mathbf{V}$$
and using $(\ref{VVt})$, we obtain
$$\mathbf{V}\begin{pmatrix}
0 & \sigma_{1,2}\alpha & \hdots & \sigma_{1,r-1} \textcolor{darkgreen}{\alpha} & \sigma_{1,r}\alpha \\
\sigma_{2,1}\alpha & 0 & \hdots & \sigma_{2,r-1}\textcolor{darkgreen}{\alpha} & \sigma_{2,r}\alpha \\
\vdots & \vdots & \ddots & \vdots & \vdots \\
\sigma_{r-1,1}\alpha & \sigma_{r-1,2}\alpha & \hdots & 0 & \sigma_{r-1,r}\alpha \\
\sigma_{r,1}\alpha & \sigma_{r,2}\alpha & \hdots & \sigma_{r,r-1}\alpha & 0
\end{pmatrix} = \frac{r}{n}\mathbf{I}_{r \times r}\mathbf{V} - \mathbf{V} = \left(\frac{r}{n} - 1\right)\mathbf{V}.$$
Dividing by $\alpha$, we have
$$\mathbf{V}\begin{pmatrix}
0 & \sigma_{1,2} & \hdots & \sigma_{1,r-1} & \sigma_{1,r} \\
\sigma_{2,1}& 0 & \hdots & \sigma_{2,r-1} & \sigma_{2,r}\\
\vdots & \vdots & \ddots & \vdots & \vdots \\
\sigma_{r-1,1}& \sigma_{r-1,2}& \hdots & 0 & \sigma_{r-1,r}\\
\sigma_{r,1}& \sigma_{r,2}& \hdots & \sigma_{r,r-1}& 0
\end{pmatrix}= \frac{1}{\alpha}\left(\frac{r}{n} - 1\right)\mathbf{V},$$
so $C = \frac{1}{\alpha}\left(\frac{r}{n} - 1\right)$.

When $d$ is even, we now show that all of $\mathbf{v}_1,\ldots,\mathbf{v}_r$ are eigenvectors of $\mathcal{T}$. We have
\[
\mathcal{T} \cdot \mathbf{v}^{d-1}_j = \sum_{i=1}^{r}\langle \mathbf{v}_i, \mathbf{v}_j\rangle^{d-1}\mathbf{v}_i = \mathbf{v}_j + \sum_{i \in [r]\setminus\{j\}} (\sigma_{i,j}\alpha)^{d-1}\mathbf{v}_i = \mathbf{v}_j + \alpha^{d-1}\sum_{i \in [r]\setminus\{j\}} \sigma_{i,j}\mathbf{v}_i
\]
$$= \mathbf{v}_{j} + \alpha^{d-1}\left(\frac{1}{\alpha}\left(\frac{r}{n}-1\right)\right)\mathbf{v}_{j} = \left(1 + \alpha^{d-2}\left(\frac{r}{n}-1\right)\right)\mathbf{v}_{j}.$$
Therefore, by \eqref{generalSpectralBound}, we obtain
\begin{align*}
\rho(\mathbf{J}(\mathbf{v}_j)) &  \le \frac{d-1}{1 + \alpha^{d-2}\left(\frac{r}{n}-1\right)}  \alpha^{d-2} \|\mathbf{V}\mathbf{V}^\top\|_2 = \frac{\frac{r}{n}\alpha^{d-2}(d-1)}{1 + \alpha^{d-2}\left(\frac{r}{n}-1\right)},
\end{align*}
where the last equality follows from \eqref{VVt}. Thus, $\mathbf{v}_j$ is a robust eigenvector for any $j \in [r]$ and all even $d$ for which the quantity above is strictly less than $1$ (which holds when $d$ is large enough).
\end{proof}
\begin{example}\textbf{(Orthogonally Decomposable Tensors)}
An ETF $\mathbf{v}_1,\ldots,\mathbf{v}_n \in \mathbb{R}^n$ with $r = n$ is clearly an orthonormal set of vectors, with constant $\alpha = 0$, and thus we may choose $\sigma_{i,j} = 1$ for all $i,j \in [r]$. Tensors $\mathcal{T}$ generated by this ETF are called \textit{orthogonally decomposable tensors} and their properties have been studied in \cite{Robeva,RobevaSeigal}. It is not hard to see that an orthogonally decomposable tensor has all of $\mathbf{v}_1,\ldots,\mathbf{v}_n$ as eigenvectors, and the spectral radius bound $(\ref{generalSpectralBound})$ is trivially equal to $0$ and therefore less than $1$, meaning all of $\mathbf{v}_1$,\ldots, $\mathbf{v}_n$ are also robust eigenvectors of $\mathcal{T}$, for all $d \geq 2$.
\end{example}
In the following sections, we first study in more detail the robust eigenvectors of tensors generated by particular ETFs, and then by an ES which is not an ETF.
\subsection{Regular Simplex Tensors}\label{sec::regularSimplex}

An ETF $\mathbf{v}_1,\ldots,\mathbf{v}_{n+1} \in \mathbb{R}^n$ with $r = n+1$ always consists of the vertices of a regular simplex in $\mathbb{R}^n$ (page 623, \cite{ETFexistence}), called a \textit{regular $n$-simplex frame}, or \textit{regular simplex frame} for short, with constant $\alpha = \frac{1}{n}$ and signs $\sigma_{i,j} = -1$ for all $i,j \in [r], i \neq j$. A particular example of a regular simplex is given by the \textit{Mercedes-Benz frame} in $\mathbb{R}^2$, as shown in Figure~\ref{MBframeFig}, consisting of the vectors
\begin{equation}\label{MBvectors}
    \mathbf{v}_1 = \begin{pmatrix} 0\\1
\end{pmatrix}, \mathbf{v}_2 = \begin{pmatrix} \frac{\sqrt 3}2\\ -\frac12\end{pmatrix}, \mathbf{v}_3 = \begin{pmatrix} -\frac{\sqrt{3}}2\\-\frac12
\end{pmatrix}.
\end{equation}

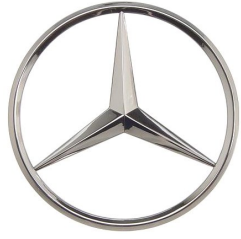
\begin{figure}
    \centering

\tikzset{every picture/.style={line width=0.75pt}} %set default line width to 0.75pt        

\begin{tikzpicture}[x=0.75pt,y=0.75pt,yscale=-1,xscale=1,scale=0.5]
%uncomment if require: \path (0,703); %set diagram left start at 0, and has height of 703

%Shape: Ellipse [id:dp7728032727854224] 
\draw  [fill={rgb, 255:red, 0; green, 0; blue, 0 }  ,fill opacity=1 ] (315.3,305.12) .. controls (317.64,305.23) and (319.47,307.18) .. (319.39,309.47) .. controls (319.31,311.75) and (317.34,313.52) .. (315,313.41) .. controls (312.66,313.29) and (310.83,311.35) .. (310.91,309.06) .. controls (311,306.77) and (312.96,305.01) .. (315.3,305.12) -- cycle ;
%Shape: Circle [id:dp5400246409930924] 
\draw  [line width=1.5]  (187,313.41) .. controls (187,242.71) and (244.31,185.41) .. (315,185.41) .. controls (385.69,185.41) and (443,242.71) .. (443,313.41) .. controls (443,384.1) and (385.69,441.41) .. (315,441.41) .. controls (244.31,441.41) and (187,384.1) .. (187,313.41) -- cycle ;
%Straight Lines [id:da5820771793182284] 
\draw [line width=2.25]    (315.15,309.26) -- (315,189.41) ;
\draw [shift={(315,185.41)}, rotate = 89.93] [color={rgb, 255:red, 0; green, 0; blue, 0 }  ][line width=2.25]    (17.49,-5.26) .. controls (11.12,-2.23) and (5.29,-0.48) .. (0,0) .. controls (5.29,0.48) and (11.12,2.23) .. (17.49,5.26)   ;
%Straight Lines [id:da38003122270390644] 
\draw [line width=2.25]    (315.15,309.26) -- (210.36,377.23) ;
\draw [shift={(207,379.41)}, rotate = 327.03] [color={rgb, 255:red, 0; green, 0; blue, 0 }  ][line width=2.25]    (17.49,-5.26) .. controls (11.12,-2.23) and (5.29,-0.48) .. (0,0) .. controls (5.29,0.48) and (11.12,2.23) .. (17.49,5.26)   ;
%Straight Lines [id:da2569192950434893] 
\draw [line width=2.25]    (315.15,309.26) -- (421.57,373.34) ;
\draw [shift={(425,375.41)}, rotate = 211.05] [color={rgb, 255:red, 0; green, 0; blue, 0 }  ][line width=2.25]    (17.49,-5.26) .. controls (11.12,-2.23) and (5.29,-0.48) .. (0,0) .. controls (5.29,0.48) and (11.12,2.23) .. (17.49,5.26)   ;

% Text Node
\draw (273,213.81) node [anchor=north west][inner sep=0.75pt]  [font=\large]  {$v_{1}$};
% Text Node
\draw (214,317.81) node [anchor=north west][inner sep=0.75pt]  [font=\large]  {$v_{2}$};
% Text Node
\draw (390,317.81) node [anchor=north west][inner sep=0.75pt]  [font=\large]  {$v_{3}$};

\end{tikzpicture}

    \caption{The Mercedes-Benz frame $(\ref{MBvectors})$ on the $xy$-plane.}
    \label{MBframeFig}
\end{figure}

We will call a tensor generated by a regular simplex frame a \textit{regular $n$-simplex tensor}, or \textit{regular simplex tensor} for short. Let $\mathbf{V} \in \mathbb{R}^{n \times (n+1)}$ be the matrix whose columns are $\mathbf{v}_1,\ldots,\mathbf{v}_{n+1}$. We observe for the regular simplex frame that
$$\mathbf{V}\mathbf{V}^\top\begin{pmatrix}
1 \\ \vdots \\ 1
\end{pmatrix} = \begin{pmatrix}
1 & -\frac{1}{n} & \hdots & -\frac{1}{n} & -\frac{1}{n} \\
-\frac{1}{n} & 1 & \hdots & -\frac{1}{n} & -\frac{1}{n} \\
\vdots & \vdots & \ddots & \vdots & \vdots \\
-\frac{1}{n} & -\frac{1}{n} & \hdots & -\frac{1}{n} & 1
\end{pmatrix}\begin{pmatrix}
1 \\ \vdots \\ 1
\end{pmatrix} = \mathbf{0},$$
and since $\mathbf{V}$ has rank $n$, the kernel of $\mathbf{V}$ is the span of $(1,\ldots,1)^\top \in \mathbb{R}^{n+1}$. Thus, consider the regular simplex tensor
\begin{equation}\label{simplexTensor}
    \mathcal{T} = \sum_{i = 1}^{n+1} \mathbf{v}_i^{\otimes d}
\end{equation}
for $d \geq 2$. \textcolor{darkgreen}{Since regular simplex tensors satisfy the conditions of both \Cref{kernelConditionLemma,constanteigenvector}, we can conclude that all of $\mathbf{v}_1,\ldots,\mathbf{v}_{n+1}$ are eigenvectors of $\mathcal{T}$.}

%\color{red} TODO: add reference? \color{black}
There is a systematic method of generating regular simplex frames in $\mathbb{R}^n$ for all $n \geq 2$ as follows: if $\mathbf{e}_1,\ldots,\mathbf{e}_n \in \mathbb{R}^n$ are the standard basis vectors and $\mathbf{1}_n = (1,\ldots,1) \in \mathbb{R}^n$, then set
\begin{equation}\label{regularSimplex}
\begin{gathered}
    \mathbf{v}_i = \sqrt{1+\frac{1}{n}}\mathbf{e}_i - \frac{1}{n^{\frac{3}{2}}}\left(\sqrt{n+1}-1\right)\mathbf{1}_n, \quad \quad i \in [n] \\
    \mathbf{v}_{n+1} = -\frac{1}{\sqrt{n}}\mathbf{1}_n.
\end{gathered}    
\end{equation}
These vectors are constructed by projecting the standard basis vectors in $\mathbb{R}^{n+1}$ onto the subspace orthogonal to the vector $\boldsymbol{1}_{n+1}$ with an appropriate rotation and rescaling.

We now present a theorem which shows that all of $\mathbf{v}_1,\ldots,\mathbf{v}_{n+1}$ are also robust eigenvectors of $\mathcal{T}$ for many values of $n$ and $d$.
\begin{theorem}\label{thm:robustness_v_k}
Let
$$\mathcal{T} = \sum_{i=1}^{n+1} \mathbf{v}_i^{\otimes d}$$
be a tensor generated by a regular simplex frame $\mathbf{v}_1,\ldots,\mathbf{v}_{n+1} \in \mathbb{R}^n$. Then all of $\mathbf{v}_1,\ldots,\mathbf{v}_{n+1}$ are robust eigenvectors for $\mathcal{T}$ for $n \geq 2$ and $d \geq 3$ such that $n + d \geq 7$.
\end{theorem}
\begin{proof}
For a regular $n$-simplex frame, we have $r = n+1$, $\alpha = \frac{1}{n}$, and $\sigma_{i,j} = -1$ for all $i,j \in [n+1]$, $i \neq j$. Since $(1,\ldots,1)^\top \in Ker(\mathbf{V})$, both bounds $(\ref{lambdasInKernelBound})$ and $(\ref{allOnesBound})$ apply. Thus,
\[
\frac{\frac{r}{n}\alpha^{d-2}(d-1)}{1 + \alpha^{d-2}\left(\frac{r}{n}-1\right)} = \frac{(n+1)(d-1)}{n^{d-1} + 1} < \frac{(n+1)(d-1)}{n^{d-1}-1} = \frac{\frac{r}{n}\alpha^{d-2}(d-1)}{\left(\min_{i \in [r]} |\lambda_i|\right)(1 - \alpha^{d-1})}.
\]
Hence, regardless of the parity of $d$, it suffices to find values of $n$ and $d$ for which $\frac{(n+1)(d-1)}{n^{d-1}-1} < 1$. This happens if and only if the following quantity is positive ($\gamma(n,d) > 0$):
\[
\gamma(n,d) = n^{d-1} +n-d-dn.
\]
We can easily check that $\gamma(n,d)$  is positive for the following values:
\[
\gamma(2,5) = 3, \quad \gamma(3,4) = 14,  \quad \gamma(4,3) = 5 
\]
Moreover the partial derivatives
\[
\frac{\partial}{\partial n}\gamma(n,d) = (d-1)(n^{d-2} -1), \quad \frac{\partial}{\partial d}\gamma(n,d) = \ln(n) n^{d-1} -(n+1)
\]
are positive for $n \ge 2$, $d \ge 3$.
This guarantees that $\gamma(n,d) > 0$ whenever
\[
(n,d) \ge (2,5) \text{ or }(n,d) \ge (3,4) \text{ or }(n,d) \ge (4,3), 
\]
and thus all of $\mathbf{v}_1,\ldots,\mathbf{v}_{n+1}$ are robust eigenvectors for $n \geq 2$ and $d \geq 3$ with $n + d \geq 7$.
\end{proof}

% \subsection{Regular $2$- and $3$-Simplex Tensors}
\subsection{Regular $2$-Simplex Tensors}
In fact, for regular $2$-simplex tensors, we can prove even stronger results compared to \Cref{thm:robustness_v_k}. The next theorem concerns not only robustness of the vectors in the regular simplex frame, but regions of convergence of the tensor power method, for tensors of order $d$ when $d \geq 4$ is even.
\begin{theorem}\label{nEqual2dEqual5Case} Let $\mathbf{v}_1,\mathbf{v}_2,\mathbf{v}_3 \in \mathbb{R}^2$ be vectors of a regular $2$-simplex frame. If $\mathbf{x}_0 \in \mathbb{R}^2$ and there is a unique $\mathbf{v} \in \{\mathbf{v}_1,-\mathbf{v}_1,\mathbf{v}_2,-\mathbf{v}_2,\mathbf{v}_3,-\mathbf{v}_3\}$ which maximizes $\langle \mathbf{v}, \mathbf{x}_0 \rangle$, then the tensor power method with initializing vector $\mathbf{x}_0$ applied to the tensor 
$$\mathcal{T}= \mathbf{v}_1^{\otimes d} + \mathbf{v}_2^{\otimes d} + \mathbf{v}_3^{\otimes d}.$$
will converge to $\mathbf{v}$, for all even $d \geq 6$. If $d = 4$, then any initializing vector $\mathbf{x}_0$ is a fixed point of the tensor power method.
\end{theorem}
We leave the proof of this theorem in the Appendix.

The results of \Cref{thm:robustness_v_k,nEqual2dEqual5Case} can be visualized in Table~\ref{convergenceTable}. 
We denote by a tick \textcolor{darkgreen}{mark (\cmark)} convergence which is guaranteed by \Cref{thm:robustness_v_k}, and by a \textcolor{darkgreen}{cross mark (\xmark)} failure of convergence for the case $n = 2$ and $d = 4$, where every initializing vector is a fixed point of the tensor power method due to \Cref{nEqual2dEqual5Case}, \textcolor{darkgreen}{and for the cases $n = 2$ and $d = 2,3$, due to Theorem 16 (i) and (iii) in \cite{conjectureFalse}. Note that \cite{conjectureFalse} appeared during the revision phase of the manuscript of this paper.}

In addition, we performed the following two numerical experiments. The first numerical experiment concerns robustness. Let $\mathcal{T}$ be the tensor $(\ref{simplexTensor})$ generated by the regular simplex frame $(\ref{regularSimplex})$. We can observe the values of $n$ and $d$ for which the tensor power method applied to $\mathcal{T}$ converges to vectors in the frame. Using MATLAB, we choose an initial vector $\mathbf{x}_0$ drawn from a uniform distribution on the unit sphere in $\mathbb{R}^n$ and apply the tensor power method $\mathbf{x}_{k+1} = \frac{\mathcal{T}(\mathbf{x}_k,\ldots,\mathbf{x}_k,\cdot)}{\|\mathcal{T}(\mathbf{x}_k,\ldots,\mathbf{x}_k,\cdot)\|}$ for $100$ iterations. \textcolor{darkgreen}{In Table~\ref{convergenceTable},} we denote by a \textcolor{darkgreen}{triangle (\ding{115})} a lack of convergence, which occurs if $\|\mathbf{x}_{100}-\mathbf{v}_j\| > 10^{-10}$ for all eigenvectors $\mathbf{v}_j$ of $\mathcal{T}$. We then performed this experiment for $2 \leq n,d \leq 10$. In the cases with the \textcolor{darkgreen}{tick} mark \textcolor{darkgreen}{(convergence guaranteed by \Cref{thm:robustness_v_k})}, we did not observe the tensor power method converging to any vectors other than $\mathbf{v}_1,\ldots,\mathbf{v}_{n+1}$. This suggests that the frame vectors may be the only robust eigenvectors, which we leave as a conjecture:
\begin{conjecture}\label{MBconjecture}
The \textcolor{darkgreen}{only} robust eigenvectors of a regular simplex tensor $(\ref{simplexTensor})$ are the vectors in the frame.
\end{conjecture}
\textcolor{darkgreen}{When $n = 2$, Theorem 12 of \cite{conjectureFalse} enumerates all possible eigenvectors of a regular simplex tensor, and Theorem 16 of \cite{conjectureFalse} shows that the robust eigenvectors are the vectors in the frame. \Cref{MBconjecture} is therefore true when $n = 2$.}

\begin{table}
\centering
\begin{tabular}{|l|l|l|l|l|l|l|l|l|l|}
\hline
\backslashbox{$n$}{$d$} & 2 & 3 & 4 & 5 & 6 & 7 & 8 & 9 & 10 \\ \hline
2   & \xmark  & \xmark  & \xmark  & \cmark  & \cmark  & \cmark  & \cmark  &  \cmark &  \cmark  \\ \hline
3   &  \ding{115}  &  \ding{115}  &  \cmark   &  \cmark & \cmark  &  \cmark &  \cmark  & \cmark  &  \cmark  \\ \hline
4   &  \ding{115}  & \cmark   & \cmark   &  \cmark  & \cmark   &  \cmark  & \cmark  &  \cmark & \cmark   \\ \hline
5   &  \ding{115} &  \cmark & \cmark  & \cmark & \cmark  &  \cmark & \cmark  & \cmark &  \cmark  \\ \hline
6   &  \ding{115} &  \cmark &  \cmark &  \cmark & \cmark  & \cmark  &  \cmark & \cmark  & \cmark   \\ \hline
7   & \ding{115}  &  \cmark & \cmark  &  \cmark & \cmark  & \cmark  & \cmark  &  \cmark &   \cmark \\ \hline
8   & \ding{115}  &  \cmark & \cmark  &  \cmark & \cmark  & \cmark  & \cmark  &  \cmark &   \cmark \\ \hline
9   & \ding{115}  &  \cmark & \cmark  &  \cmark & \cmark  & \cmark  & \cmark  &  \cmark &   \cmark    \\ \hline
10  & \ding{115}  &  \cmark & \cmark  &  \cmark & \cmark  & \cmark  & \cmark  &  \cmark &   \cmark    \\ \hline
\end{tabular}
\caption{Convergence of the tensor power method for the tensor $\mathcal{T} = \sum_{i = 1}^{n+1} \mathbf{v}_i^{\otimes d}$ generated by the regular simplex frame $(\ref{regularSimplex})$.
\textcolor{darkgreen}{Tick mark (\cmark): convergence guaranteed by \Cref{thm:robustness_v_k}; cross mark (\xmark): failure of convergence when $n = 2$ for $d = 4$, (every initializing vector is a fixed point, see \Cref{nEqual2dEqual5Case}) and $d = 2,3$ (see Theorem 16 (i) and (iii) in \cite{conjectureFalse}); triangle (\ding{115}): lack of convergence in the numerical experiment.}}\label{convergenceTable}
% \caption{Verification of the convergence of the tensor power method for the tensor $\mathcal{T} = \sum_{i = 1}^{n+1} \mathbf{v}_i^{\otimes d}$ generated by the regular simplex frame $(\ref{regularSimplex})$.}
% \label{verificationTable}
\end{table}

% The second numerical experiment concerns the regions of convergence of the tensor power method in Theorems \ref{nEqual2dEqual5Case} and \ref{nEqual3dEqual4Case}.

The second numerical experiment concerns the regions of convergence of the tensor power method in \Cref{nEqual2dEqual5Case}. Figure~\ref{PowerMethod2D} shows regions of convergence to the eigenvectors $\mathbf{v}_1$, $\mathbf{v}_2$, and $\mathbf{v}_3$, starting from an initial vector $\mathbf{x}_0$ of the tensor power method applied to the regular simplex tensor $\mathcal{T}$ in $(\ref{simplexTensor})$ generated by the vectors in a Mercedes-Benz frame $(\ref{MBvectors})$. For $\mathbf{x}_0$ in the blue, red, and green regions, the method will convergence to $\mathbf{v}_1$, $\mathbf{v}_2$, and $\mathbf{v}_3$, respectively. As \Cref{nEqual2dEqual5Case} predicts, when $d \geq 6$ is even, the regions of convergence form a partition of the unit disk into sectors. One can also observe a fractal subdivision of the regions of convergence for odd values of $d$, which we lack an explanation for. As a consequence of \Cref{mainTheorem}, however, larger values of $d$ result in greater robustness of the eigenvectors, and hence the observed thinning of these fractal subdivisions.

\begin{figure}
  \centering
\[
\begin{array}{ccc}
d = 6 & d = 8 & d = 10\\
    \includegraphics[width=0.3\textwidth]{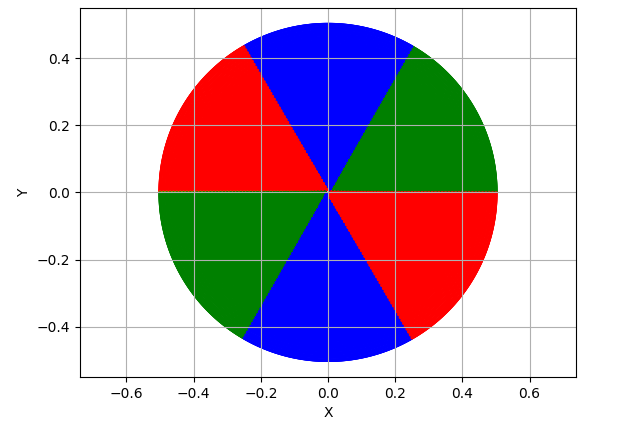}&
    \includegraphics[width=0.3\textwidth]{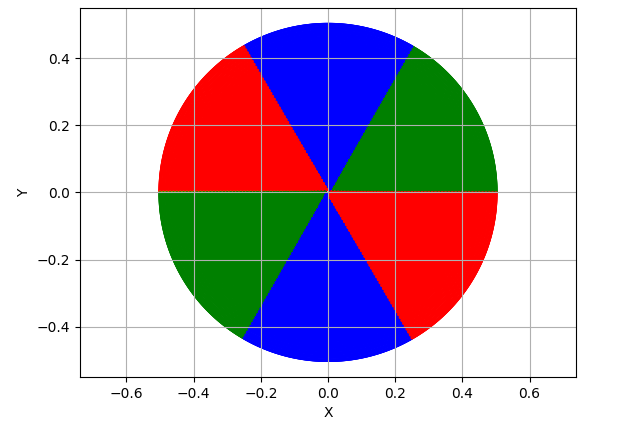}&
	\includegraphics[width=0.3\textwidth]{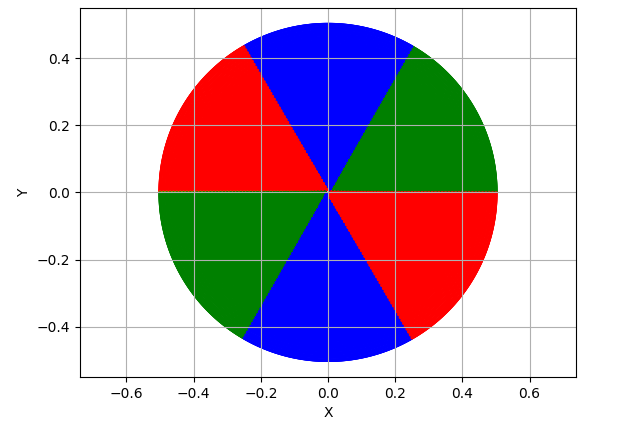} \\
&&\\
d = 5 & d = 7 & d = 9\\
	\includegraphics[width=0.3\textwidth]{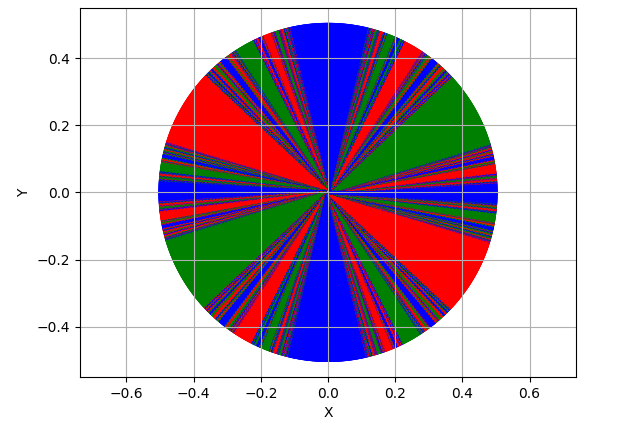} &
    \includegraphics[width=0.3\textwidth]{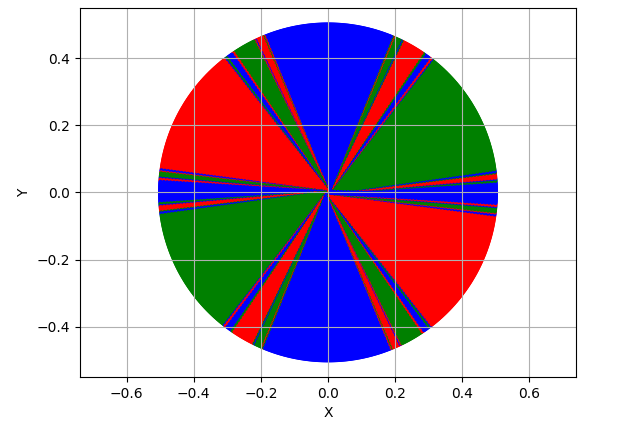}&
	\includegraphics[width=0.3\textwidth]{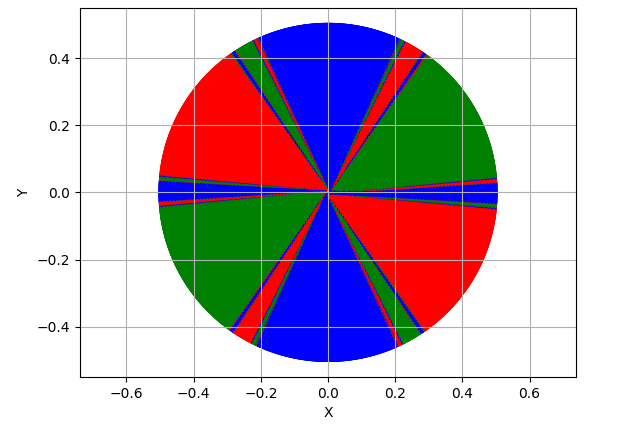}
\end{array}
\]
     \caption{Regions of convergence on the unit disk in $\mathbb{R}^2$ of the tensor power method for different values of $d$ of the Mercedez-Benz tensor.}\label{PowerMethod2D}
\end{figure}

For the tensor $\mathcal{T}$ generated by a Mercedes-Benz frame ($n = 2$) for $3 \leq d \leq 10$ given by the vectors in $(\ref{MBvectors})$, we can find its eigenvectors and their corresponding eigenvalues and multiplicities, which we show in Tables \ref{MBtensorEigenvectors1} and \ref{MBtensorEigenvectors2}.

\begin{table}
\centering
% \vspace*{-0.5in}
\begin{tabular}{|cccc|}
\hline
$d$                                        & eigenvector                                           & eigenvalue                                & multiplicity \\ \hline
\multicolumn{1}{|l|}{\multirow{3}{*}{$3$}} & $(0,1)^\top$                                          & $\frac{3}{4}$                             & $1$          \\ \cline{2-4} 
\multicolumn{1}{|l|}{}                     & $(\frac{\sqrt{3}}{2},-\frac{1}{2})^\top$              & $\frac{3}{4}$                             & $1$          \\ \cline{2-4} 
\multicolumn{1}{|l|}{}                     & $(-\frac{\sqrt{3}}{2},-\frac{1}{2})^\top$             & $\frac{3}{4}$                             & $1$          \\ \hline
\multicolumn{1}{|l|}{\multirow{4}{*}{$4$}} & $(0,1)^\top$                                          & $\frac{9}{8}$                             & $1$          \\ \cline{2-4} 
\multicolumn{1}{|l|}{}                     & $(\frac{\sqrt{3}}{2},-\frac{1}{2})^\top$              & $\frac{9}{8}$                             & $1$          \\ \cline{2-4} 
\multicolumn{1}{|l|}{}                     & $(-\frac{\sqrt{3}}{2},-\frac{1}{2})^\top$             & $\frac{9}{8}$                             & $1$          \\ \cline{2-4} 
\multicolumn{1}{|l|}{}                     & $(1,0)^\top$                                          & $\frac{9}{8}$                             & $1$          \\ \hline
\multicolumn{1}{|l|}{\multirow{5}{*}{$5$}} & $(0,1)^\top$                                          & $\frac{15}{16}$                           & $1$          \\ \cline{2-4} 
\multicolumn{1}{|l|}{}                     & $(\frac{\sqrt{3}}{2},-\frac{1}{2})^\top$              & $\frac{15}{16}$                           & $1$          \\ \cline{2-4} 
\multicolumn{1}{|l|}{}                     & $(-\frac{\sqrt{3}}{2},-\frac{1}{2})^\top$             & $\frac{15}{16}$                           & $1$          \\ \cline{2-4} 
\multicolumn{1}{|l|}{}                     & $(-\frac{\sqrt{2}i}{2},\frac{\sqrt{2}}{2})^\top$      & $\frac{3\sqrt{2}}{8}$                     & $1$          \\ \cline{2-4} 
\multicolumn{1}{|l|}{}                     & $(\frac{\sqrt{2}i}{2},\frac{\sqrt{2}}{2})^\top$       & $\frac{3\sqrt{2}}{8}$                     & $1$          \\ \hline
\multicolumn{1}{|l|}{\multirow{6}{*}{$6$}} & $(0,1)^\top$                                          & $\frac{33}{32}$                           & $1$          \\ \cline{2-4} 
\multicolumn{1}{|l|}{}                     & $(\frac{\sqrt{3}}{2},-\frac{1}{2})^\top$              & $\frac{33}{32}$                           & $1$          \\ \cline{2-4} 
\multicolumn{1}{|l|}{}                     & $(-\frac{\sqrt{3}}{2},-\frac{1}{2})^\top$             & $\frac{33}{32}$                           & $1$          \\ \cline{2-4} 
\multicolumn{1}{|l|}{}                     & $(\frac{1}{2},\frac{\sqrt{3}}{2})^\top$               & $\frac{27}{32}$                           & $1$          \\ \cline{2-4} 
\multicolumn{1}{|l|}{}                     & $(-\frac{1}{2},\frac{\sqrt{3}}{2})^\top$              & $\frac{27}{32}$                           & $1$          \\ \cline{2-4} 
\multicolumn{1}{|l|}{}                     & $(1,0)^\top$                                          & $\frac{27}{32}$                           & $1$          \\ \hline
\multicolumn{1}{|l|}{\multirow{5}{*}{$7$}} & $(0,1)^\top$                                          & $\frac{63}{64}$                           & $1$          \\ \cline{2-4} 
\multicolumn{1}{|l|}{}                     & $(\frac{\sqrt{3}}{2},-\frac{1}{2})^\top$              & $\frac{63}{64}$                           & $1$          \\ \cline{2-4} 
\multicolumn{1}{|l|}{}                     & $(-\frac{\sqrt{3}}{2},-\frac{1}{2})^\top$             & $\frac{63}{64}$                           & $1$          \\ \cline{2-4} 
\multicolumn{1}{|l|}{}                     & $(-\frac{\sqrt{2}i}{2},\frac{\sqrt{2}}{2})^\top$      & $0$                                       & $2$          \\ \cline{2-4} 
\multicolumn{1}{|l|}{}                     & $(\frac{\sqrt{2}i}{2},\frac{\sqrt{2}}{2})^\top$       & $0$                                       & $2$          \\ \hline
\multicolumn{1}{|l|}{\multirow{8}{*}{$8$}} & $(0,1)^\top$                                          & $\frac{129}{128}$                         & $1$          \\ \cline{2-4} 
\multicolumn{1}{|l|}{}                     & $(\frac{\sqrt{3}}{2},-\frac{1}{2})^\top$              & $\frac{129}{128}$                         & $1$          \\ \cline{2-4} 
\multicolumn{1}{|l|}{}                     & $(-\frac{\sqrt{3}}{2},-\frac{1}{2})^\top$             & $\frac{129}{128}$                         & $1$          \\ \cline{2-4} 
\multicolumn{1}{|l|}{}                     & $(-\frac{\sqrt{2}i}{2},\frac{\sqrt{2}}{2})^\top$      & $\frac{3}{16}$                            & $1$          \\ \cline{2-4} 
\multicolumn{1}{|l|}{}                     & $(\frac{\sqrt{2}i}{2},\frac{\sqrt{2}}{2})^\top$        & $\frac{3}{16}$                            & $1$          \\ \cline{2-4} 
\multicolumn{1}{|l|}{}                     & $(-\frac{1}{2},\frac{\sqrt{3}}{2})^\top$              & $\frac{81}{128}$                          & $1$          \\ \cline{2-4} 
\multicolumn{1}{|l|}{}                     & $(\frac{1}{2},\frac{\sqrt{3}}{2})^\top$              & $\frac{81}{128}$                          & $1$          \\ \cline{2-4} 
\multicolumn{1}{|l|}{}                     & $(1,0)^\top$                                          & $\frac{81}{128}$                          & $1$          \\ \hline
\multicolumn{1}{|l|}{\multirow{9}{*}{$9$}} & $(0,1)^\top$                                          & $\frac{255}{256}$                         & $1$          \\ \cline{2-4} 
\multicolumn{1}{|l|}{}                     & $(\frac{\sqrt{3}}{2},-\frac{1}{2})^\top$              & $\frac{255}{256}$                         & $1$          \\ \cline{2-4} 
\multicolumn{1}{|l|}{}                     & $(-\frac{\sqrt{3}}{2},-\frac{1}{2})^\top$              & $\frac{255}{256}$                         & $1$          \\ \cline{2-4} 
\multicolumn{1}{|l|}{}                     & $\mathbf{x} \approx (1/Z)(0.393942-0.624439i,1)^\top$ & $\lambda_1 \approx -0.234194 - 0.107117i$ & $1$          \\ \cline{2-4} 
\multicolumn{1}{|l|}{}                     & $\mathbf{y} \approx (1/Z)(1.965672i,1)^\top$          & $\lambda_2 \approx 0.257529$              & $1$          \\ \cline{2-4} 
\multicolumn{1}{|l|}{}                     & $\overline{\mathbf{x}}$                               & $\overline{\lambda_1}$                    & $1$          \\ \cline{2-4} 
\multicolumn{1}{|l|}{}                     & $(-\mathbf{x}_1,\mathbf{x}_2)^\top$                   & $\lambda_1$                               & $1$          \\ \cline{2-4} 
\multicolumn{1}{|l|}{}                     & $(-\overline{\mathbf{x}_1},\mathbf{x}_2)^\top$        & $\overline{\lambda_1}$                    & $1$          \\ \cline{2-4} 
\multicolumn{1}{|l|}{}                     & $(-\overline{\mathbf{y}_1},\mathbf{y}_2)^\top$        & $\lambda_2$                               & $1$          \\ \hline
\end{tabular}
% \caption{Eigenvectors of $\mathcal{T}$ generated by $(\ref{MBvectors})$, where $Z \in \mathbb{R}$ is a normalizing constant.}\label{MBtensorEigenvectors1}
% \caption{The unit-norm eigenvectors and their eigenvalues with multiplicities of the tensor $\mathcal{T}$ generated by the vectors $(\ref{MBvectors})$, where $Z \in \mathbb{R}$ is a normalizing constant.}\label{MBtensorEigenvectors1}
\caption{The unit-norm eigenvectors and their eigenvalues with multiplicities of the tensor $\mathcal{T} = \sum_{i = 1}^3 \left((0,1)^\top\right)^{\otimes d} + \left(\left(\sqrt 3/2,-1/2\right)^\top\right)^{\otimes d} + \left(\left(-\sqrt{3}/2,-1/2\right)^\top\right)^{\otimes d}$, where $Z \in \mathbb{R}$ is a normalizing constant.}\label{MBtensorEigenvectors1}
\end{table}

\begin{table}
\centering
\begin{tabular}{|cccc|}
\hline
$d$                                         & eigenvector                                      & eigenvalue        & multiplicity \\ \hline
\multicolumn{1}{|l|}{\multirow{8}{*}{10}} & $(0,1)^\top$                                     & $\frac{513}{512}$ & $1$          \\ \cline{2-4} 
\multicolumn{1}{|l|}{}                    & $(\frac{\sqrt{3}}{2},-\frac{1}{2})^\top$         & $\frac{513}{512}$ & $1$          \\ \cline{2-4} 
\multicolumn{1}{|l|}{}                    & $(-\frac{\sqrt{3}}{2},-\frac{1}{2})^\top$        & $\frac{513}{512}$ & $1$          \\ \cline{2-4} 
\multicolumn{1}{|l|}{}                    & $(-\frac{\sqrt{2}i}{2},\frac{\sqrt{2}}{2})^\top$ & $0$               & $2$          \\ \cline{2-4} 
\multicolumn{1}{|l|}{}                    & $(\frac{\sqrt{2}i}{2},\frac{\sqrt{2}}{2})^\top$  & $0$               & $2$          \\ \cline{2-4} 
\multicolumn{1}{|l|}{}                    & $(-\frac{1}{2},\frac{\sqrt{3}}{2})^\top$         & \textcolor{darkgreen}{$\frac{243}{512}$}               & $1$          \\ \cline{2-4} 
\multicolumn{1}{|l|}{}                    & $(\frac{1}{2},\frac{\sqrt{3}}{2})^\top$          & $\frac{243}{512}$ & $1$          \\ \cline{2-4} 
\multicolumn{1}{|l|}{}                    & $(1,0)^\top$                                     & $\frac{243}{512}$ & $1$          \\ \hline
\end{tabular}
\caption{Continued from Table \ref{MBtensorEigenvectors1}.}\label{MBtensorEigenvectors2}
\end{table}

The eigenvectors can be found by solving \textcolor{darkgreen}{a system} of two polynomial equations.
\textcolor{darkgreen}{We are looking for eigenvectors $(u,v)^\top \in \mathbb{C}^2$ satisfying $u^2 + v^2 = 1$ (i.e., having the unit norm if the eigenvector is real).
In order to be an eigenvector, $(u,v)^\top$ should be collinear to $\mathcal{T} \cdot \left((u,v)^\top\right)^{d-1}$.
These conditions can be encoded by the following equations:}

\medskip

\begin{equation*}
%\hspace*{-0.375in}
\begin{cases}
\text{det}\begin{pmatrix}
\frac{\sqrt{3}}{2}\left(\frac{\sqrt{3}}{2}u - \frac{1}{2}v\right)^{d-1} -  \frac{\sqrt{3}}{2}\left(-\frac{\sqrt{3}}{2}u - \frac{1}{2}v\right)^{d-1} & u \\
v^{d-1} -\frac{1}{2}\left(\frac{\sqrt{3}}{2}u - \frac{1}{2}v\right)^{d-1} - \frac{1}{2}\left(-\frac{\sqrt{3}}{2}u - \frac{1}{2}v\right)^{d-1} & v\end{pmatrix} = 0, \\
\hfil u^2 + v^2 = 1,
\end{cases}    
\end{equation*}
\textcolor{darkgreen}{where the elements in the first column in the determinant are exactly the elements of $\mathcal{T} \cdot \left((u,v)^\top\right)^{d-1}$.}
%where the first equation encodes the information that $(u,v)^\top \in \mathbb{C}^2$ is an eigenvector of $\mathcal{T}$, and the second equation imposes that $(u,v)^\top$, the unique representative of an eigenvector, has norm $1$.

The multiplicity of an eigenvector is the multiplicity of its solution in this system of equations. In \cite{CartwrightSturmfels}, it was shown that if a $n \times \cdots \times n$ ($d$ times) tensor with complex entries has finitely many eigenvectors in $\mathbb{C}^n$ up to scaling, then this number is $\frac{(d-1)^n-1}{d-2}$ counted with multiplicity. In this setting, we have $\frac{(d-1)^n - 1}{d-2} = d$, and as can be seen in the tables, the sum of the multiplicities of each eigenvector for each $d$ is equal to $d$, meaning that these tables form a complete list of the eigenvectors of $\mathcal{T}$. The table also suggests another conjecture:
\begin{conjecture}\label{complexConjecture}
All eigenvectors of a regular simplex tensor are complex, except for the vectors in the frame, and when $d \geq 6$ and even, the vectors on the boundary of the regions of convergence.
\end{conjecture}

\textcolor{darkgreen}{This conjecture has been proven in the case $n = 2$ due to Theorem 12 (ii) and (iii) in \cite{conjectureFalse}.}

\textcolor{darkgreen}{We see from Tables \ref{MBtensorEigenvectors1} and \ref{MBtensorEigenvectors2} that when $d \geq 6$ and even, the vectors $(\frac{1}{2},\frac{\sqrt{3}}{2})^\top$, $(1,0)^\top$, and $(-\frac{1}{2},\frac{\sqrt{3}}{2})^\top$ appear as eigenvectors, which lie exactly on the boundary of the regions of convergence seen in Figure~\ref{PowerMethod2D}. Figure~\ref{PowerMethod2D} suggests that these eigenvectors are not robust.}

\textcolor{darkgreen}{Whether the only robust eigenvectors of a general ES or ETF tensor are the vectors generating the tensor, or under what conditions this occurs, is left as an open problem in \Cref{openProblem}.}

\subsection{Other Equiangular Tensors}
If $\mathbf{V} \in \mathbb{R}^{n \times r}$ is the matrix of an ETF whose columns are $\mathbf{v}_1,\ldots,\mathbf{v}_r \in \mathbb{R}^n$, then more examples of ETFs include the diagonals of a cube in $\mathbb{R}^3$,
$$\mathbf{V} = \frac{1}{\sqrt{3}}\begin{pmatrix}
1 & -1 & -1 & -1 \\
1 & 1 & -1 & 1 \\
1 & 1 & 1 & -1
\end{pmatrix},$$
the diagonals of a regular icosahedron in $\mathbb{R}^3$ (Figure~\ref{icosahedron}),
\begin{equation}\label{icosahedralFrame}
    \mathbf{V} = \frac{1}{\sqrt{1 + \varphi^2}}\begin{pmatrix}
0 & 0 & 1 & -1 & \varphi & -\varphi \\
1 & -1 & \varphi & \varphi & 0 & 0 \\
\varphi & \varphi & 0 & 0 & 1 & 1
\end{pmatrix},
\end{equation}
where $\varphi = \frac{1 + \sqrt{5}}{2}$, and the following $16$ equiangular lines in $\mathbb{R}^6$,
$$\mathbf{V} = \frac{1}{\sqrt{6}}\left(\begin{smallmatrix}
1 & -1 & -1 & -1 & -1 & -1 & 1 & 1 & 1 & 1 & 1 & 1 & 1 & 1 & 1 & 1 \\
1 & -1 & 1 & 1 & 1 & 1 & -1 & -1 & -1 & -1 & 1 & 1 & 1 & 1 & 1 & 1 \\
1 & 1 & -1 & 1 & 1 & 1 & -1 & 1 & 1 & 1 & -1 & -1 & -1 & 1 & 1 & 1 \\
1 & 1 & 1 & -1 & 1 & 1 & 1 & -1 & 1 & 1 & -1 & 1 & 1 & -1 & -1 & 1 \\
1 & 1 & 1 & 1 & -1 & 1 & 1 & 1 & -1 & 1 & 1 & -1 & 1 & -1 & 1 & -1 \\
1 & 1 & 1 & 1 & 1 & -1 & 1 & 1 & 1 & -1 & 1 & 1 & -1 & 1 & -1 & -1
\end{smallmatrix}\right).$$

\begin{figure}
    \centering
    \includegraphics[width=0.225\textwidth]{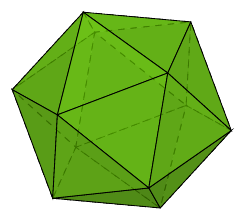}
    \caption{A regular icosahedron in $\mathbb{R}^3$.  \phantom{\cite{pict}}}
    \label{icosahedron}
\end{figure}

\begin{example}\textbf{(Icosahedral Frame)}
Consider the tensor $\mathcal{T}$ generated by the icosahedral frame $(\ref{icosahedralFrame})$,
$$\mathcal{T} = \sum_{i = 1}^6 \mathbf{v}_i^{\otimes d},$$
for $d \geq 2$. Here, $n = 3$, $r = 6$, $\alpha = \frac{1}{\sqrt{5}}$, and the Gram matrix is
\begin{equation}\label{gramMatrixIcosaedron}
\mathbf{V}^\top\mathbf{V} = \begin{pmatrix}
1 & \alpha & \alpha & \alpha & \alpha & \alpha \\
\alpha & 1 & -\alpha & -\alpha & \alpha & \alpha \\
\alpha & -\alpha & 1 & \alpha & \alpha & -\alpha \\
\alpha & -\alpha & \alpha & 1 & -\alpha & \alpha \\
\alpha & \alpha & \alpha & -\alpha & 1 & -\alpha \\
\alpha & \alpha & -\alpha & \alpha & -\alpha & 1
\end{pmatrix}.    
\end{equation}
By \Cref{constanteigenvector}, all of $\mathbf{v}_1,\ldots,\mathbf{v}_6$ are robust eigenvectors of $\mathcal{T}$ for all even $d \geq 6$ where the bound
\begin{equation}\label{icosahedralBound}
    \frac{\frac{r}{n}\alpha^{d-2}(d-1)}{1 + \alpha^{d-2}\left(\frac{r}{n}-1\right)} = \frac{2(d-1)\left(\frac{1}{\sqrt{5}}\right)^{d-2}}{1 + \left(\frac{1}{\sqrt{5}}\right)^{d-2}}
\end{equation}
from $(\ref{allOnes})$ is strictly less than $1$. However, when $d$ is odd, one can verify that indeed, $\mathbf{v}_2,\ldots,\mathbf{v}_6$ are not eigenvectors of $\mathcal{T}$. But $\mathbf{v}_1$ is still an eigenvector of $\mathcal{T}$ because $\sigma_{1,1} = \cdots = \sigma_{6,1}$ (they are all equal to $1$) as seen from the Gram matrix $(\ref{gramMatrixIcosaedron})$. Thus, the bound $(\ref{allOnes})$ still applies for $\mathbf{v}_1$ and $(\ref{icosahedralBound})$ is strictly less than $1$ for all $d \geq 5$, regardless of the parity of $d$, and hence for all these values of $d$, $\mathbf{v}_1$ is a robust eigenvector.
\end{example}
\begin{example}\textbf{(An ES that is not an ETF)}
Below are $6$ equiangular lines in $\mathbb{R}^4$, forming an ES that is not an ETF:
$$\mathbf{V} = \begin{pmatrix}
1 & \frac{1}{3} & \frac{1}{3} & \frac{1}{3} & \frac{1}{3} & \frac{1}{3} \\
0 & \frac{2\sqrt{2}}{3} & -\frac{\sqrt{2}}{3} & -\frac{\sqrt{2}}{3} & -\frac{\sqrt{2}}{3} & -\frac{\sqrt{2}}{3} \\
0 & 0 & \frac{\sqrt{6}}{3} & 0 & -\frac{\sqrt{6}}{3} & 0 \\
0 & 0 & 0 & \frac{\sqrt{6}}{3} & 0 & -\frac{\sqrt{6}}{3}
\end{pmatrix}.$$
Then, \Cref{constanteigenvector} fails, and indeed, if $\mathcal{T}$ is the following tensor generated by this ES,
$$\mathcal{T} = \sum_{i = 1}^6 \mathbf{v}_i^{\otimes d},$$
then none of $\mathbf{v}_1,\ldots,\mathbf{v}_6$ are eigenvectors of $\mathcal{T}$ for all $d \geq 2$. But if $\mathcal{T}$ is the following tensor,
$$\mathcal{T} = \sum_{i = 1}^6 \lambda_i\mathbf{v}_i^{\otimes d},$$
where $(\lambda_1,\ldots,\lambda_6) \in Ker(\mathbf{V})$, then \Cref{kernelConditionLemma} still holds with $\mu_j = \lambda_j$ and $d$ odd, so that all of $\mathbf{v}_1,\ldots,\mathbf{v}_6$ are eigenvectors. When $d$ is even, $\mathbf{v}_1$ is an eigenvector with $\mu_j = \lambda_j$ since $\sigma_{1,1} = \cdots = \sigma_{6,1}$ (they are all equal to $1$), while $\mathbf v_2,\ldots, \mathbf v_6$ are not.
\end{example}

\section{Conclusion}\label{conclusionSection}
We have given sufficient conditions for when an eigenvector in a symmetric decomposition of a tensor is robust. We have then explored robustness for a family of tensors that are generated by sets of vectors that are mutually equiangular. We leave behind an open problem whose solution would entirely solve the task of characterising eigenvectors that can be retrieved by the tensor power method:
\begin{problem}{\label{openProblem}}
Under what conditions are the elements $\mathbf{v}_1,\ldots,\mathbf{v}_r$ of an ES \textcolor{darkgreen}{the only} real, robust eigenvectors of a tensor $\mathcal{T}$ they generate? Is there an efficient method to distinguish the vectors of the ES and the other (if any) robust eigenvectors of $\mathcal{T}$?
\end{problem}
In addition, it would be interesting to extend our results to tensors generated by non-equiangular sets.

In \cite{MuHsuGoldfarb} it was shown that the eigenvectors of orthogonally decomposable tensors are \textit{stable} under perturbation of the tensor, i.e., if a tensor with small Frobenius norm is added to an orthogonally decomposable tensor, then the eigenvectors of this new tensor can be retrieved by the tensor power method and they will be close to the original eigenvectors. From other numerical experiments we performed, we strongly believe that this is also true for equiangular tight frame decomposable tensors:

\begin{conjecture}\label{stablilityConjecture}
The eigenvectors of an equiangular tight frame decomposable tensor are stable under perturbation and can be approximately recovered using the tensor power method.
\end{conjecture}

\section*{Acknowledgements}
We wish to thank Jeffery Zhang and Kevin Shen for creating Figure~\ref{PowerMethod2D}.

\appendix
\section{Proofs for Mercedes-Benz frames}
\begin{proof}[Proof of \Cref{nEqual2dEqual5Case}] Suppose $\mathbf{v}_1,\mathbf{v}_2,\mathbf{v}_3$ is the Mercedes-Benz frame $(\ref{MBvectors})$. Let $\mathbf{x}_0 \in \mathbb{R}^2$ be an initializing vector for the tensor power method applied to $\mathcal{T}$. Then
\begin{align*}
\mathbf{x}_{k+1} & := \textcolor{darkgreen}{\mathcal{T} \cdot \mathbf{x}_k^{d-1} =
\mathbf{v}_1 (\langle \mathbf{x}_k, \mathbf{v}_1 \rangle)^{d-1}+
\mathbf{v}_2 (\langle \mathbf{x}_k, \mathbf{v}_2 \rangle)^{d-1}+
\mathbf{v}_3 (\langle \mathbf{x}_k, \mathbf{v}_3 \rangle)^{d-1}} \\
 &= \begin{pmatrix}
           0 \\
           1
         \end{pmatrix} c_k^{d-1}
        +
        \begin{pmatrix}
           \frac{\sqrt{3}}{2} \\
           -\frac{1}{2}
         \end{pmatrix} a_k^{d-1}
        +
        \begin{pmatrix}
           -\frac{\sqrt{3}}{2} \\
           -\frac{1}{2}
         \end{pmatrix} b_k^{d-1} = \begin{pmatrix}
           \frac{\sqrt{3}}{2}(a_k^{d-1} - b_k^{d-1}) \\
           c_k^{d-1} - \frac{1}{2}(a_k^{d-1} + b_k^{d-1})
         \end{pmatrix},
\end{align*}
where
$$a_k := \langle \mathbf{x}_k, \mathbf{v}_2 \rangle, \quad
    b_k := \langle \mathbf{x}_k, \mathbf{v}_3 \rangle, \quad
    c_k := \langle \mathbf{x}_k, \mathbf{v}_1 \rangle.$$
\textcolor{darkgreen}{It is easy to see that the only unit-norm vector $\mathbf{x} \in \mathbb{R}^2$ which satisfies both $\frac{\langle \mathbf{x},\mathbf{v}_2 \rangle}{\langle \mathbf{x},\mathbf{v}_1 \rangle} = -\frac{1}{2}$ and $\frac{\langle \mathbf{x},\mathbf{v}_3 \rangle}{\langle \mathbf{x},\mathbf{v}_1 \rangle} = -\frac{1}{2}$ is $\mathbf{x} = \mathbf{v}_1 = (0,1)^\top$ or $\mathbf{x} = -\mathbf{v}_1 = (0,-1)^\top$. Therefore,} if we can find a neighbourhood of $\mathbf{v}_1$ (or $-\mathbf{v}_1$)  for $\mathbf{x}_0$ to lie in so that
$$\lim_{k \rightarrow \infty} \frac{\langle \mathbf{x}_{k+1},\mathbf{v}_2\rangle}{\langle \mathbf{x}_{k+1},\mathbf{v}_1 \rangle} = -\frac{1}{2}, \quad \lim_{k \rightarrow \infty} \frac{\langle \mathbf{x}_{k+1},\mathbf{v}_3\rangle}{\langle \mathbf{x}_{k+1},\mathbf{v}_1\rangle} = -\frac{1}{2},$$
then we will have proven that the tensor power method converges to $\mathbf{v}_1$ (or $-\mathbf{v}_1$) because this ratio of inner products is invariant under scaling of $\mathbf{x}_{k+1}$. Inner products are also invariant under rotations on the plane, and therefore this will also prove the robustness of $\mathbf{v}_2$ and $\mathbf{v}_3$ and for any regular simplex frame in $\mathbb{R}^2$, not just the Mercedez-Benz frame, since we can rotate a frame so that $\mathbf{x}_0$ lies in a neighbourhood of $\mathbf{v}_1$ (or $-\mathbf{v}_1$). We see that
$$\frac{\langle \mathbf{x}_{k+1},\mathbf{v}_2\rangle}{\langle \mathbf{x}_{k+1},\mathbf{v}_1 \rangle} = \frac{\frac{3}{4}(a_k^{d-1} - b_k^{d-1}) - \frac{1}{2}c_k^{d-1} + \frac{1}{4}(a_k^{d-1} + b_k^{d-1})}{c_k^{d-1} - \frac{1}{2}(a_k^{d-1} + b_k^{d-1})} = \frac{a_k^{d-1} - \frac{1}{2}b_k^{d-1} - \frac{1}{2}c_k^{d-1}}{-\frac{1}{2}a_k^{d-1} - \frac{1}{2}b_k^{d-1} + c_k^{d-1}}$$

\begin{equation}\label{alphaSubstitute}
= \frac{(\frac{a_k}{c_k})^{d-1} - \frac{1}{2}(\frac{b_k}{c_k})^{d-1} - \frac{1}{2}}{-\frac{1}{2}(\frac{a_k}{c_k})^{d-1} - \frac{1}{2}(\frac{b_k}{c_k})^{d-1} + 1} = \frac{(-\alpha_k)^{d-1} - \frac{1}{2}(-(1-\alpha_k))^{d-1} - \frac{1}{2}}{1 - \frac{1}{2}(-\alpha_k)^{d-1} - \frac{1}{2}(-(1-\alpha_k))^{d-1}} = -\alpha_{k+1}
\end{equation}

where we define
$$\alpha_k := -\frac{a_k}{c_k}, \quad 1 - \alpha_k = \frac{c_k + a_k}{c_k} = -\frac{b_k}{c_k}.$$

Therefore, we need to find a neighbourhood of $\alpha_0$ around $\frac{1}{2}$ for which $\lim_{k \rightarrow \infty} \alpha_k = \frac{1}{2}$. If we can do this, then we will also have proven $\lim_{k \rightarrow \infty} \frac{\langle \mathbf{x}_{k+1},\mathbf{v}_3\rangle}{\langle \mathbf{x}_{k+1},\mathbf{v}_1\rangle} = -\frac{1}{2}$ since $\frac{\langle \mathbf{x}_{k+1},\mathbf{v}_3\rangle}{\langle \mathbf{x}_{k+1},\mathbf{v}_1\rangle} = -(1 - \alpha_{k+1})$.

Note that the hypothesis in the theorem about the uniqueness of a maximizer ensures that $\alpha_0 \neq 0$. We claim that the neighbourhood where $0 < \alpha_0 < 1$ will suffice. This is the neighbourhood $S$ of $\mathbf{v}_1$ and $-\mathbf{v}_1$ where for every $\mathbf{x}_0 \in S$, $|\langle \mathbf{v}, \mathbf{x}_0 \rangle|$ is maximized among $\mathbf{v} \in \{\mathbf{v}_1,-\mathbf{v}_1,\mathbf{v}_2,-\mathbf{v}_2,\mathbf{v}_3,-\mathbf{v}_3\}$ by $\mathbf{v} = \pm\mathbf{v}_1$. Thus, assuming $0 < \alpha_0 < 1$, $0 < 1 - \alpha_0 < 0$, and since
\begin{equation}\label{alphaFraction}
    \alpha_{k+1} = -\frac{(-\alpha_k)^{d-1} - \frac{1}{2}(-(1-\alpha_k))^{d-1} - \frac{1}{2}}{1 + \frac{1}{2}\alpha_k^{d-1} + \frac{1}{2}(1-\alpha_k)^{d-1}} = \frac{\alpha_k^{d-1} - \frac{1}{2}(1-\alpha_k)^{d-1} + \frac{1}{2}}{1 + \frac{1}{2}\alpha_k^{d-1} + \frac{1}{2}(1-\alpha_k)^{d-1}}
\end{equation}
when $d$ is even, by the induction, the numerator and denominator are both positive, and hence $0 < \alpha_{k+1}$. Furthermore,

$$0 < \frac12 - \frac12\alpha_k^{d-1} + (1-\alpha_k)^{d-1}$$
by induction, and hence
$$\alpha_k^{d-1} - \frac12(1-\alpha_k)^{d-1} + \frac12 < 1 + \frac12\alpha_k^{d-1} + \frac12(1-\alpha_k)^{d-1},$$
which shows that $\alpha_{k+1} < 1$ by $(\ref{alphaFraction})$, and thus $0 < \alpha_k < 1$ for all $k$.

Now we show that $\left(\alpha_{k+1} - \frac12\right) - \left(\alpha_k-\frac12\right)  = C_{d,k}\left(\alpha_k - \frac12\right)$, where $C_{d,k}$ is a continuous function of $\alpha_k$, with $-1 < C_{d,k} < 0$ when $d \geq 6$ is even. We will also show that $C_{d,k} = 0$ if $\alpha_k = 0$ or $\alpha_k = 1$. Thus, assuming $d \geq 6$ is even and $0 < \alpha_k < 1$, there then exists a constant, namely $C_d = \sup_{k \geq 0} C_{d,k}$, such that $-1 < C_{d,k} \leq C_d < 0$ for all $k$. Thus, if we can prove these statements, then rearranging the equality, we would obtain $\left|\alpha_{k+1} - \frac{1}{2}\right| = C^{\prime}_{d,k}\left|\alpha_k - \frac{1}{2}\right|$ where $0 < C^{\prime}_{d,k} \leq C_d + 1 < 1$, which would imply that $\lim_{k \rightarrow \infty} \alpha_k = \frac{1}{2}$ when $d \geq 6$ and even. When $d = 4$, we will see that $C_{d,k} = 0$ and hence every initializing vector $\mathbf{x}_0$ is a fixed point of the tensor power method when $d = 4$. We have
\begin{align*}
&\left(\alpha_{k+1} - \frac12\right) - \left(\alpha_k - \frac12\right) = \left(\frac{\alpha_k^{d-1} - \frac12(1-\alpha_k)^{d-1} + \frac12}{1 + \frac12\alpha_k^{d-1} + \frac12(1-\alpha_k)^{d-1}} - \frac12\right) - \left(\alpha_k - \frac12\right)\\
&= \left(\frac{2\alpha_k^{d-1} - (1-\alpha_k)^{d-1} + 1 - 1 - \frac12\alpha_k^{d-1} - \frac12(1-\alpha_k)^{d-1}}{2(1 + \frac12\alpha_k^{d-1} + \frac12(1-\alpha_k)^{d-1})}\right) - \left(\alpha_k - \frac12\right)\\
&= \left(\frac{\frac32\alpha_k^{d-1} - \frac32(1-\alpha_k)^{d-1}}{2(1 + \frac12\alpha_k^{d-1} + \frac12(1-\alpha_k)^{d-1})}\right) - \left(\alpha_k - \frac12\right)\\
&= \frac{\frac32\alpha_k^{d-1} - \frac32(1-\alpha_k)^{d-1} - 2\alpha_k - \alpha_k^d - \alpha_k(1-\alpha_k)^{d-1} + 1 + \frac12\alpha_k^{d-1} + \frac12(1-\alpha_k)^{d-1}}{2 + \alpha_k^{d-1} + (1-\alpha_k)^{d-1}}\\
&=\frac{2\alpha_k^{d-1} - (1-\alpha_k)^{d-1} - 2\alpha_k + 1 - \alpha_k^d - \alpha_k(1-\alpha_k)^{d-1}}{2 + \alpha_k^{d-1} + (1-\alpha_k)^{d-1}}\\
&=\left( \frac{-2 + 4\sum_{j = 0}^{d-2}\alpha_k^{d-2-j}(1 - \alpha_k)^j - 2(1-\alpha_k)^{d-1}-2\alpha_k\sum_{j = 0}^{d-2}\alpha_k^{d-2-j}(1 - \alpha_k)^j}{2 + \alpha_k^{d-1} + (1-\alpha_k)^{d-1}}\right)\left(\alpha_k-\frac12\right)\\
&= 2\frac{(2-\alpha_k)\sum_{j = 0}^{d-2}\alpha_k^{d-2-j}(1 - \alpha_k)^j -(1+(1-\alpha_k)^{d-1})}{2 + \alpha_k^{d-1} + (1-\alpha_k)^{d-1}}\left(\alpha_k-\frac12\right) = C_{d,k}\left(\alpha_k-\frac12\right)\\
&=2(2-\alpha_k)\frac{\sum_{j = 0}^{d-2}\alpha_k^{d-2-j}(1 - \alpha_k)^j - \sum_{j = 0}^{d-2}(-1)^j(1 - \alpha_k)^j}{2 + \alpha_k^{d-1} + (1-\alpha_k)^{d-1}}\left(\alpha_k-\frac12\right).
\end{align*}
Notice that $C_{d,k}$ is a continuous function of $\alpha_k$ and when $\alpha_k = 0$ or $\alpha_k = 1$, $C_{d,k} = 0$. We observe that
\begin{align*}
&\sum_{j = 0}^{d-2}\alpha_k^{d-2-j}(1 - \alpha_k)^j - \sum_{j = 0}^{d-2}(-1)^j(1 - \alpha_k)^j \\
&= \sum_{j = 0}^{\frac{d-4}{2}}\left(\alpha_k^{d-2-2j}(1 - \alpha_k)^{2j} + \alpha_k^{d-2-(2j+1)}(1 - \alpha_k)^{2j+1}\right) + (1 - \alpha_k)^{d-2}\\
& \phantom{=} - \sum_{j = 0}^{\frac{d-4}{2}}\left((1 - \alpha_k)^{2j} - (1 - \alpha_k)^{2j+1}\right)  - (1 - \alpha_k)^{d-2}\\
&= \sum_{j = 0}^{\frac{d-4}{2}}\alpha_k^{d-2-(2j+1)}(1 - \alpha_k)^{2j} - \sum_{j = 0}^{\frac{d-4}{2}}\alpha_k(1 - \alpha_k)^{2j} \\
&= \sum_{j = 0}^{\frac{d-4}{2}} \left(\alpha_k^{d-2-(2j+1)} - \alpha_k\right)(1 - \alpha_k)^{2j}.
\end{align*}
When $d = 4$, this sum is $0$ and hence $C_{d,k} = 0$. When $d \geq 6$ and even, each term has $\left(\alpha_k^{d-2-(2j+1)} - \alpha_k\right)(1 - \alpha_k)^{2j} < 0$, and thus $C_{d,k} < 0$.

Starting with the inequality
$$\left((2\alpha_k - 1) + 2(2 - \alpha_k)\right)\sum_{j = 0}^{d-2}\alpha_k^{d-2-j}(1 - \alpha_k)^j = 3\sum_{j = 0}^{d-2}\alpha_k^{d-2-j}(1 - \alpha_k)^j > 0,$$
we find that
\begin{align*}
&\left(\alpha_k^{d-1} - (1 - \alpha_k)^{d-1}\right) + 2(2 - \alpha_k)\sum_{j = 0}^{d-2}\alpha_k^{d-2-j}(1 - \alpha_k)^j > 0 \\
& 2\left((2 - \alpha_k)\sum_{j = 0}^{d-2}\alpha_k^{d-2-j}(1 - \alpha_k)^j - \left(1 + (1 - \alpha_k)^{d-1}\right)\right) > -2 - \alpha_k^{d-1} - (1 - \alpha_k)^{d-1}\\
&C_{d,k} = 2\frac{(2-\alpha_k)\sum_{j = 0}^{d-2}\alpha_k^{d-2-j}(1 - \alpha_k)^j -(1+(1-\alpha_k)^{d-1})}{2 + \alpha_k^{d-1} + (1-\alpha_k)^{d-1}} > -1
\end{align*}
when $d \geq 6$ and even.
\end{proof}
\bibliographystyle{siamplain}
\bibliography{main}

\begin{thebibliography}{10}

\bibitem{Anandkumar2014}
{\sc A.~Anandkumar, R.~Ge, S.~K. D.~Hsu, and M.~Telgarsky}, {\em Tensor
  decompositions for learning latent variable models}, Journal of Machine
  Learning Research, 15 (2014), pp.~2773--2832.

\bibitem{pmlr-v40-Anandkumar15}
{\sc A.~Anandkumar, R.~Ge, and M.~Janzamin}, {\em Learning overcomplete latent
  variable models through tensor methods}, in Proceedings of The 28th
  Conference on Learning Theory, P.~Grünwald, E.~Hazan, and S.~Kale, eds.,
  vol.~40 of Proceedings of Machine Learning Research, Paris, France, 03--06
  Jul 2015, PMLR, pp.~36--112,
  \url{https://proceedings.mlr.press/v40/Anandkumar15.html}.

\bibitem{BDHR}
{\sc A.~Boralevi, J.~Draisma, E.~Horobet, and E.~Robeva}, {\em Orthogonal and
  unitary tensor decomposition from an algebraic perspective}, Israel Journal
  of Mathematics, 222 (2017), p.~223–260.

\bibitem{BRAMBILLA20081229}
{\sc M.~C. Brambilla and G.~Ottaviani}, {\em On the {Alexander–Hirschowitz}
  theorem}, Journal of Pure and Applied Algebra, 212 (2008), pp.~1229--1251,
  \url{https://doi.org/https://doi.org/10.1016/j.jpaa.2007.09.014},
  \url{https://www.sciencedirect.com/science/article/pii/S0022404907002344}.

\bibitem{CartwrightSturmfels}
{\sc D.~Cartwright and B.~Sturmfels}, {\em The number of eigenvalues of a
  tensor}, Linear Algebra and its Applications, 438 (2013), pp.~942--952,
  \url{https://doi.org/https://doi.org/10.1016/j.laa.2011.05.040},
  \url{https://www.sciencedirect.com/science/article/pii/S0024379511004629}.
\newblock Tensors and Multilinear Algebra.

\bibitem{chen2009simax}
{\sc J.~Chen and Y.~Saad}, {\em On the tensor {SVD} and the optimal low rank
  orthogonal approximation of tensors}, SIAM Journal on Matrix Analysis and
  Applications, 30 (2009), pp.~1709--1734,
  \url{https://doi.org/10.1137/070711621},
  \url{https://doi.org/10.1137/070711621},
  \url{https://arxiv.org/abs/https://doi.org/10.1137/070711621}.

\bibitem{etfWelchBound}
{\sc O.~Christensen, S.~Datta, and R.~Y. Kim}, {\em Equiangular frames and
  generalizations of the {Welch} bound to dual pairs of frames}, Linear and
  Multilinear Algebra, 68 (2020), pp.~2495--2505,
  \url{https://doi.org/10.1080/03081087.2019.1586825},
  \url{https://doi.org/10.1080/03081087.2019.1586825},
  \url{https://arxiv.org/abs/https://doi.org/10.1080/03081087.2019.1586825}.

\bibitem{pict}
{\sc Icosahedron}, {\em Wikimedia~Commons}, 2007,
  \url{https://commons.wikimedia.org/wiki/File:Icosahedron.svg}.
\newblock File: Icosahedron.svg.

\bibitem{comon2009tensor}
{\sc P.~Comon, X.~Luciani, and A.~L. De~Almeida}, {\em Tensor decompositions,
  alternating least squares and other tales}, Journal of Chemometrics: A
  Journal of the Chemometrics Society, 23 (2009), pp.~393--405.

\bibitem{conjectureFalse}
{\sc A.~Czaplinski, T.~Raasch, and J.~Steinberg}, {\em Real eigenstructure of
  regular simplex tensors}, 2022,
  \url{https://doi.org/10.48550/ARXIV.2203.01865},
  \url{https://arxiv.org/abs/2203.01865}.

\bibitem{DeLathauwer1997}
{\sc L.~De~Lathauwer}, {\em Signal processing based on multilinear algebra},
  PhD thesis, Katholike Universiteit Leuven, 1997.

\bibitem{DeLathauwer1995}
{\sc L.~De~Lathauwer, P.~Comon, B.~De~Moor, and J.~Vandewalle}, {\em
  Higher-order power method, application in {I}ndependent {C}omponent
  {A}nalysis}, in NOLTA Conference, vol.~1, Las Vegas, 10--14 Dec 1995,
  pp.~91--96.

\bibitem{DeLathauwer2000}
{\sc L.~{De Lathauwer}, B.~{De Moor}, and J.~Vandewalle}, {\em On the best
  rank-1 and rank-$(r_1, r_2 ,\ldots,r_n)$ approximation of higher-order
  tensors}, SIAM Journal on Matrix Analysis and Applications, 21 (2000),
  pp.~1324--1342.

\bibitem{steinerFrames}
{\sc M.~Fickus, D.~G. Mixon, and J.~C. Tremain}, {\em Steiner equiangular tight
  frames}, Linear Algebra and its Applications, 436 (2012), pp.~1014--1027,
  \url{https://doi.org/https://doi.org/10.1016/j.laa.2011.06.027},
  \url{https://www.sciencedirect.com/science/article/pii/S0024379511004769}.

\bibitem{friedland2013best}
{\sc S.~Friedland}, {\em Best rank one approximation of real symmetric tensors
  can be chosen symmetric}, Frontiers of Mathematics in China, 8 (2013),
  pp.~19--40.

\bibitem{tensorTrains}
{\sc K.~Halaseh, T.~Muller, and E.~Robeva}, {\em Orthogonal decomposition of
  tensor trains}, Linear and Multilinear Algebra, 0 (2021), pp.~1--31,
  \url{https://doi.org/10.1080/03081087.2021.1965947},
  \url{https://doi.org/10.1080/03081087.2021.1965947},
  \url{https://arxiv.org/abs/https://doi.org/10.1080/03081087.2021.1965947}.

\bibitem{Jennrich}
{\sc R.~A. Harshman}, {\em {F}oundations of the {P}{A}{R}{A}{F}{A}{C}
  procedure: {M}odels and conditions for an "explanatory" multi-modal factor
  analysis}, UCLA Working Papers in Phonetics, 16 (1970), pp.~1--84.

\bibitem{HL}
{\sc C.~Hillar and L.-H. Lim}, {\em Most tensor problems are {NP}-hard},
  Journal of the ACM, 60 (2013).

\bibitem{Hitchcock}
{\sc F.~L. Hitchcock}, {\em The expression of a tensor or a polyadic as a sum
  of products}, Journal of Mathematics and Physics, 6 (1927), pp.~164--189.

\bibitem{hu2018convergence}
{\sc S.~Hu and G.~Li}, {\em Convergence rate analysis for the higher order
  power method in best rank one approximations of tensors}, Numerische
  Mathematik, 140 (2018), pp.~993--1031.

\bibitem{equiangularLineBEST}
{\sc Z.~Jiang, J.~Tidor, Y.~Yao, S.~Zhang, and Y.~Zhao}, {\em Equiangular lines
  with a fixed angle}, Annals of Mathematics, 194 (2021), pp.~729--743.

\bibitem{kileel2021subspace}
{\sc J.~Kileel and J.~M. Pereira}, {\em Subspace power method for symmetric
  tensor decomposition and generalized pca}, 2019.
\newblock arXiv preprint arXiv:1912.04007.

\bibitem{kofidis2002simax}
{\sc E.~Kofidis and P.~A. Regalia}, {\em On the best rank-1 approximation of
  higher-order supersymmetric tensors}, SIAM Journal on Matrix Analysis and
  Applications, 23 (2002), pp.~863--884,
  \url{https://doi.org/10.1137/S0895479801387413},
  \url{https://doi.org/10.1137/S0895479801387413},
  \url{https://arxiv.org/abs/https://doi.org/10.1137/S0895479801387413}.

\bibitem{kolda2015symmetric}
{\sc T.~G. Kolda}, {\em Symmetric orthogonal tensor decomposition is trivial},
  (2015).
\newblock arXiv preprint arXiv:1503.01375.

\bibitem{KoBa09}
{\sc T.~G. Kolda and B.~W. Bader}, {\em Tensor decompositions and
  applications}, SIAM Review, 51 (2009), pp.~455--500,
  \url{https://doi.org/10.1137/07070111X}.

\bibitem{shiftedPower}
{\sc T.~G. Kolda and J.~R. Mayo}, {\em Shifted power method for computing
  tensor eigenpairs}, SIAM Journal on Matrix Analysis and Applications, 32
  (2011), pp.~1095--1124, \url{https://doi.org/10.1137/100801482},
  \url{https://doi.org/10.1137/100801482},
  \url{https://arxiv.org/abs/https://doi.org/10.1137/100801482}.

\bibitem{li2018jacobi}
{\sc J.~Li, K.~Usevich, and P.~Comon}, {\em Globally convergent {Jacobi}-type
  algorithms for simultaneous orthogonal symmetric tensor diagonalization},
  SIAM Journal on Matrix Analysis and Applications, 39 (2018), pp.~1--22.

\bibitem{LHLim}
{\sc L.-H. Lim}, {\em Singular values and eigenvalues of tensors: a variational
  approach}, Proceedings of IEEE Workshop on Computational Advances in
  Multisensor Adaptive Processing, 1 (2005), pp.~129--132.

\bibitem{MuHsuGoldfarb}
{\sc C.~Mu, D.~Hsu, and D.~Goldfarb}, {\em Successive rank-one approximations
  for nearly orthogonally decomposable symmetric tensors}, SIAM J. Matrix Anal.
  Appl., 36 (2015), pp.~1638--1659.

\bibitem{nie2014simax}
{\sc J.~Nie and L.~Wang}, {\em Semidefinite relaxations for best rank-1 tensor
  approximations}, SIAM Journal on Matrix Analysis and Applications, 35 (2014),
  pp.~1155--1179, \url{https://doi.org/10.1137/130935112},
  \url{https://doi.org/10.1137/130935112},
  \url{https://arxiv.org/abs/https://doi.org/10.1137/130935112}.

\bibitem{OEDING2016125}
{\sc L.~Oeding, E.~Robeva, and B.~Sturmfels}, {\em Decomposing tensors into
  frames}, Advances in Applied Mathematics, 73 (2016), pp.~125--153,
  \url{https://doi.org/https://doi.org/10.1016/j.aam.2015.10.004},
  \url{https://www.sciencedirect.com/science/article/pii/S0196885815001177}.

\bibitem{qi2005eigenvalues}
{\sc L.~Qi}, {\em Eigenvalues of a real supersymmetric tensor}, Journal of
  Symbolic Computation, 40 (2005), pp.~1302--1324.

\bibitem{Rheinboldt1998}
{\sc W.~C. Rheinboldt}, {\em Methods for Solving Systems of Nonlinear
  Equations}, Society for Industrial and Applied Mathematics, 1998,
  \url{https://doi.org/10.1137/1.9781611970012},
  \url{https://epubs.siam.org/doi/abs/10.1137/1.9781611970012},
  \url{https://arxiv.org/abs/https://epubs.siam.org/doi/pdf/10.1137/1.9781611970012}.

\bibitem{Robeva}
{\sc E.~Robeva}, {\em Orthogonal decomposition of symmetric tensors}, SIAM
  Journal on Matrix Analysis and Applications, 37 (2016), pp.~86--102,
  \url{https://doi.org/10.1137/140989340}.

\bibitem{RobevaSeigal}
{\sc E.~Robeva and A.~Seigal}, {\em Singular vectors of orthogonally
  decomposable tensors}, Linear and Multilinear Algebra, 65 (2017),
  pp.~2457--2471, \url{https://doi.org/10.1080/03081087.2016.1277508},
  \url{https://doi.org/10.1080/03081087.2016.1277508},
  \url{https://arxiv.org/abs/https://doi.org/10.1080/03081087.2016.1277508}.

\bibitem{quantumFrames}
{\sc A.~J. Scott}, {\em Tight informationally complete quantum measurements},
  Journal of Physics A, 39 (2006), pp.~13507--13530.

\bibitem{Sorber2013}
{\sc L.~Sorber, M.~Van~Barel, and L.~De~Lathauwer}, {\em Optimization-based
  algorithms for tensor decompositions: Canonical polyadic decomposition,
  decomposition in rank-\$(l\_r,l\_r,1)\$ terms, and a new generalization},
  SIAM Journal on Optimization, 23 (2013), pp.~695--720,
  \url{https://doi.org/10.1137/120868323}.

\bibitem{grassmannFrames}
{\sc T.~Strohmer and R.~W. Heath}, {\em Grassmannian frames with applications
  to coding and communication}, Applied and Computational Harmonic Analysis, 14
  (2003), pp.~257--275,
  \url{https://doi.org/https://doi.org/10.1016/S1063-5203(03)00023-X},
  \url{https://www.sciencedirect.com/science/article/pii/S106352030300023X}.

\bibitem{ETFexistence}
{\sc M.~A. Sustik, J.~A. Tropp, I.~S. Dhillon, and R.~W. Heath}, {\em On the
  existence of equiangular tight frames}, Linear Algebra and its Applications,
  426 (2007), pp.~619--635.

\bibitem{gramMatrixRank}
{\sc J.~Tropp}, {\em Complex equiangular tight frames}, Proceedings of SPIE -
  The International Society for Optical Engineering, 5914 (2005),
  \url{https://doi.org/10.1117/12.618821}.

\bibitem{uschmajew2015pjo}
{\sc A.~Uschmajew}, {\em A new convergence proof for the higher-order power
  method and generalizations}, Pac. J. Optim., 11 (2015), pp.~309--321.

\bibitem{WALDRON20092228}
{\sc S.~Waldron}, {\em On the construction of equiangular frames from graphs},
  Linear Algebra and its Applications, 431 (2009), pp.~2228--2242,
  \url{https://doi.org/https://doi.org/10.1016/j.laa.2009.07.016},
  \url{https://www.sciencedirect.com/science/article/pii/S002437950900370X}.

\bibitem{ZhaGol}
{\sc T.~Zhang and G.~H. Golub}, {\em Rank-one approximation to high order
  tensors}, SIAM J. Matrix Anal. Appl., 23 (2001), pp.~534--550.

\end{thebibliography}
\end{document}